\documentclass[11pt]{article}
 
 \pdfoutput=1
 
\usepackage[final]{epsfig}
\usepackage{ulem}
\usepackage{ltxtable}
\usepackage{amsmath}
\usepackage{amsfonts}
\usepackage{subfig}

\def\proof{\noindent {\bf Proof. }}

\def\P{\mathcal{P}}
\newcommand{\PS}{\mathbb{P}}
\newcommand{\R}{\mathbb{R}}  
\newcommand{\poly}{\mathcal{P}}
\newcommand{\spoly}{\overline{\mathcal{P}}}
\newcommand{\blocks}{\mathcal{B}}
\newcommand{\holes}{\mathcal{H}}
\newcommand{\polyblock}{B_{\mathcal{P}}}
\newcommand{\polyhole}{H_{\mathcal{P}}}
\newcommand{\spolyblock}{B_{\overline{\mathcal{P}}}}
\newcommand{\spolyhole}{H_{\overline{\mathcal{P}}}}
\newcommand{\p}{\boldsymbol{p}} 
 
\newcommand{\B}{\mathcal{B}} 
\newcommand{\Ho}{\mathcal{H}} 
\newcommand{\Su}{\mathcal{T}} 
 
\newcommand{\q}{\boldsymbol{q}} 
\newcommand{\as}{\boldsymbol{a}} 
\newcommand{\bs}{\boldsymbol{b}} 
 
\newcommand{\GM}{G^{M}} 
\newcommand{\Db}{\boldsymbol{D}} %bold 
\newcommand{\Lb}{\boldsymbol{L}} 
\newcommand{\Bb}{\boldsymbol{B}} 
\newcommand{\0}{\boldsymbol{0}} 
\newcommand{\ijkm}{<v_{i},v_{j};F^{k},F^{m}>} 
\newcommand{\jimk}{<v_{j},v_{i};F^{m},F^{k}>} 
\newcommand{\SB}{\boldsymbol{S}} 
\newcommand{\ov}[1]{\overline{#1}}

\newsavebox{\cuadrito}
\sbox{\cuadrito}{\framebox[7pt]{ }}
\newcommand{\qed}{\makebox[8pt]{}\hfill {\usebox{\cuadrito}}\medskip}

\newtheorem{theorem}{THEOREM}[section]
\newtheorem{proposition}[theorem]{PROPOSITION}
\newtheorem{corollary}[theorem]{COROLLARY}
\newtheorem{lemma}[theorem]{LEMMA}

\newtheorem{definition}{Definition}[section]

\newtheorem{conjecture}{CONJECTURE}[section]

\newcounter{example}
  \newenvironment{example}{\refstepcounter{example}
    \subsubsection{Example
     }}{$\qed$ \\}

\newcounter{mycomplaints}
\def\complaint#1{\refstepcounter{mycomplaints}%
\ifhmode%
\unskip%
{\dimen1=\baselineskip \divide\dimen1 by 2 %
\raise\dimen1\llap{\tiny -\themycomplaints-}}\fi%
\marginpar{\tiny [\themycomplaints]: #1}}%

\begin{document}

\title{The Rigidity of Spherical Frameworks: Swapping Blocks and Holes}
\author{
{Wendy Finbow \thanks{Department of Mathematics, St. Mary's University. 
The author was supported by NSERC (Canada) and York University
}}
\\
{Elissa  Ross
\thanks{Department of Mathematics and Statistics, York University. 
The author was supported in part under a grant from  NSERC (Canada).}}
\\
{ Walter Whiteley
\thanks{Department of Mathematics and Statistics, York University.
The author was supported in part by a grant 
from  NSERC (Canada).}}
}
\maketitle

\begin{abstract}  A significant range of geometric structures whose rigidity is explored for both practical and theoretical purposes are formed by modifying generically isostatic triangulated spheres.  In the block and hole structures $(\poly, \p)$, some edges are removed to make holes,  and other edges are added to create rigid sub-structures called blocks. Previous work noted a combinatorial analogy in which blocks and holes played equivalent roles.  In this paper, we connect stresses in such a structure $(\poly, \p)$ to first-order motions in a swapped structure $(\spoly, \p)$, where holes become blocks and blocks become holes.   When the initial structure is geometrically isostatic, this shows that the swapped structure is also geometrically isostatic, giving the strongest possible correspondence.  We use a projective geometric presentation of the statics and the motions, to make the key underlying correspondences transparent.
%
%A significant range of geometric structures whose rigidity
%is explored for both practical and theoretical purposes are formed by
%modifying generically isostatic triangulated spheres.  In the block and hole
%structures $(\poly,\p)$, some edges are removed (to  
%make holes) and others are added (to create rigid sub-structures called blocks).   
%Previous work noted a combinatorial analogy in which blocks and holes  
%played equivalent roles.   In  
%this paper, we connect stresses in such a structure $(\poly,\p)$ to
%first-order motions in a swapped structure  $(\ov\poly,\p)$, where  
%holes become blocks and blocks become holes.   When the initial  
%structure is geometrically isostatic, this shows that the swapped
%structure is also geometrically isostatic, giving the strongest  
%possible correspondence.   We use a projective presentation of the  
%statics and the motions, to make the key underlying correspondences
%transparent.
\\	

\\


\end{abstract}

%%%%%%%%%%%%%%%%%%%%%%%%%%%%%%%%%%%%%%%
\section{Introduction } \label{sec:introduction}
%%%%%%%%%% Walter Draft August 15%%%%%
%%%%%%%%%% revised August 17 Wendy 
%%%%%%%%%%  revised Elissa August 25
%%%%%%%%% revised Walter August 31 .......
%%%%%%%%% revised wendy Dec 15
%%%%%%%%   revised Walter February 17, October 2
%%%%%%%%%%%%%%%%%%%%%%%%%%%%%%%%%%%%%%%

The general problem of which graphs can be realized in 3-space as an isostatic
(rigid and independent) bar and joint framework is a major unsolved problem in
rigidity theory \cite{graver,wchapter}.  In the absence of a general characterization,
it becomes significant to investigate certain classes of graphs and to confirm the rigidity
of almost all realizations of the graphs in 3-space (generic rigidity). 

Historically, from the work of Cauchy and Dehn,  \cite{cauchy,dehn},
we know that convex triangulated spheres are isostatic,
and therefore, any generic realization of the corresponding 3-connected planar graphs
with $|E|=3|V|-6$ is also isostatic \cite{gluck}. In a previous paper, \cite{wfsw}, 
two of the authors considered the following process: 
remove some edges of a convex triangulated sphere (creating holes), 
insert the same number of edges to create isostatic subpieces (blocks) in the sphere, and
leave the `untouched' portions as triangulated surfaces.  
The authors verified that under certain conditions this process preserves the generically isostatic (rigid, independent) nature of 
such {\it block and hole polyhedra}.
\begin{figure}
\begin{center}
\includegraphics[width=4.5in]{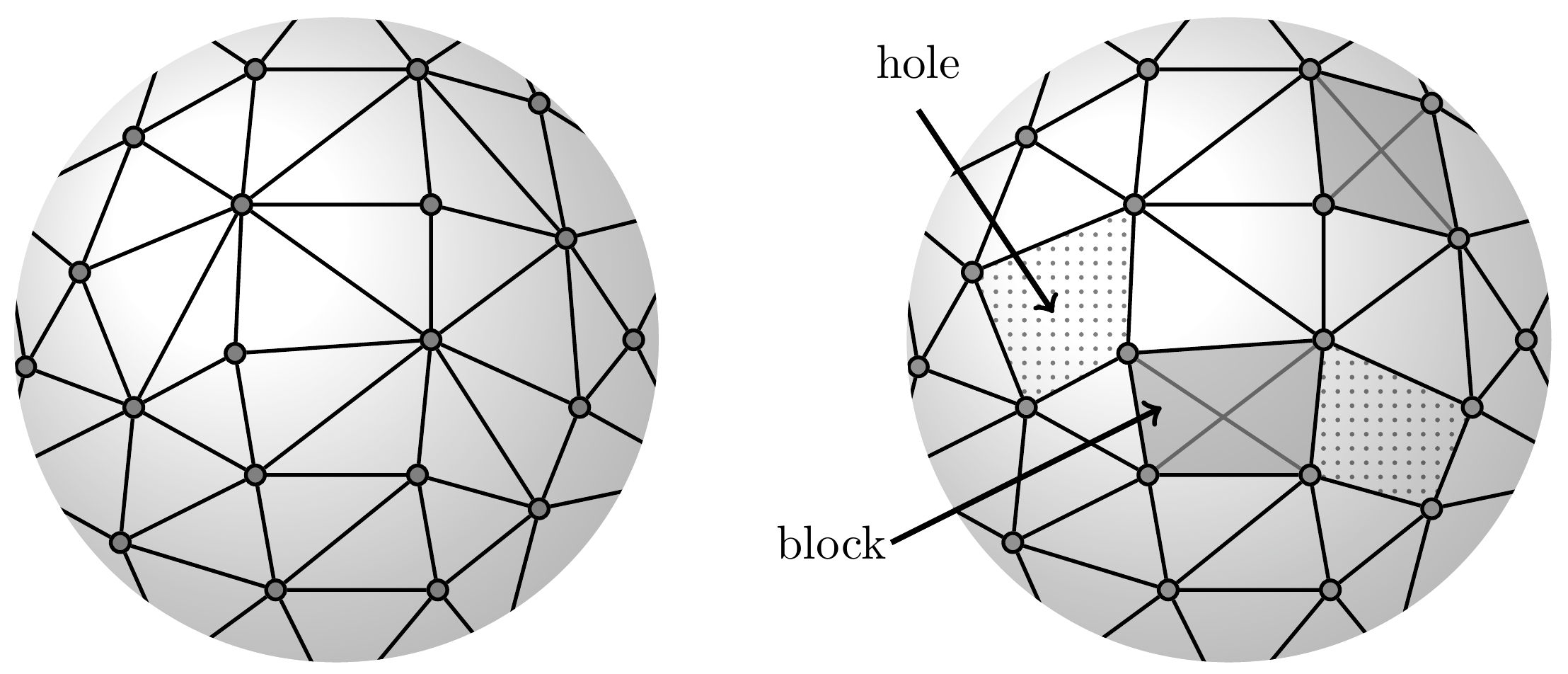}
\caption{A triangulated sphere (left). Removing some edges creates {\it holes} (dotted), and replacing the edges elsewhere creates {\it blocks} (shaded). \label{fig:sphericalPoly}}
\end{center}
\end{figure}

A read through the paper~\cite{wfsw} suggests that the blocks and holes play
`dual' roles.  That is, if we look at the necessary conditions and the conjectures for isostatic block and hole
frameworks in that paper, 
we can swap the faces which were blocks into holes and faces which were holes into blocks. This suggests that if a
block and hole framework is generically isostatic, then the {\it swapped} framework is also generically isostatic.

In this paper, we show a stronger set of geometric results that have these observed combinatorial connections as corollaries.
Let $\poly$ be a block and hole polyhedron, and let $\spoly$ be the swapped polyhedron (blocks and holes interchanged).
Let $G(\poly)$ and $G(\spoly)$ represent the graphs of the polyhedron and swapped polyhedron respectively. Let $\p$ be
an embedding of the graph into $\R^3$, and let $G(\poly, \p)$ and $G(\spoly, \p)$ be the embedded frameworks of 
$\poly$ and $\spoly$ respectively. 
We show that:
\begin{enumerate}
\item  if a block and hole polyhedral framework $G(\poly, \p)$ has a non-trivial first-order motion, then
the swapped block and hole structure $G(\spoly, \p)$ has a static self-stress in the same configuration; and
\item  if a block and hole polyhedral framework $G(\poly, \p)$ has a static self-stress, then
the swapped block and hole structure  $G(\spoly, \p)$ has a non-trivial first-order motion in the same configuration.
\end{enumerate} 

\noindent From these basic results we conclude that a block and hole polyhedron is geometrically isostatic (rigid,
independent) at a given configuration $\p$ if and only if the swapped block and hole polyhedron is also isostatic at $\p$. The generic results also follow: a block and hole polyhedron is generically
isostatic if, and only if the swapped polyhedron is generically isostatic. 

The methods used are an extension of previous geometric work connecting the first-order motions of spherical
polyhedra formed with all faces rigid and edges as hinges, and 
the stresses on a framework, that is, the stresses on the underlying graph of vertices and edges of the polyhedron \cite{crapowhiteley}.  
If the entire underlying polyhedron has only triangular faces, the correspondence is also implicit in earlier work 
on Alexandrov's Theorem \cite{infp1}.   Here we provide an overarching projective geometric analysis that 
gives a general theory extending all of these previous results.

In Section \ref{sec:background} we provide a formal introduction to our object of study, the block and hole polyhedron, and the basic
projective geometric theory of static rigidity and infinitesimal rigidity in space. We also provide a brief introduction to the projective
geometry and the language of the Grassmann-Cayley algebra in which we are presenting the rigidity. 
This projective presentation makes the correspondence much more transparent than an alternate Euclidean 
presentation would be.  It also highlights the advantages of placing infinitesimal and static rigidity 
into projective geometry: both increased simplicity and increased generality.   

In Section \ref{sec:motionsAndStressesSeparated}, we prove the main result for {\it separated} block and hole polyhedra, which is a simplified setting where we assume no pair of holes or blocks share a vertex. In section \ref{sec:gussets} we introduce {\it gussets} to account for remaining cases. 
Finally, in Section \ref{sec:extensions} we outline some extensions and  discuss further
implications of this work.  In particular, it is always of interest to determine 
which configurations make a
generically isostatic graph into a geometrically isostatic framework.  
We review some connections of this work
with prior work on these polynomial pure conditions \cite{wwI}.

%%%%%%%%%%%%%%%%%%%
\section{Background} \label{sec:background}

We begin by introducing the basic combinatorial object for this study.  This will 
be followed by a general introduction to the projective geometric theory of the
statics of frameworks and the projective geometric theory of the (first-order) motions
of hinge structures.   In the final subsection we will bring these three pieces together 
to give the notation and background for the following sections. 

%%%%%%%%%%%%%%%%%%%%%%%%%%%%%%%%%%%%%
\subsection{Block and Hole Polyhedra} \label{sec:blockAndHolePolyhedra}
%%%%%%%  Walter Thurs Aug 16  %%
%%%%%%%  Some revision Wendy Aug 17
%%%%%%%  Elissa August 25 
%%%%%%%  Wendy dec 13 
%%%%%%% Walter October 02
%%%%%%%%%%%%%%%%%%%%%%%%%%%%%%%%%%%%

In \cite{wfsw,infp2}, we introduced a construction process for block and hole polyhedra.  This construction permitted us
to extract the generic static rigidity properties of the framework of the block and hole polyhedra from the generic
behaviour of an underlying `base' block and hole polyhedron.  

In this paper, a number of the construction details from \cite{wfsw} and \cite{infp2} will not be relevant.   
Thus, we give a simplified definition of a block and hole polyhedra, starting with the definition of an abstract spherical
polyhedron.

An {\it abstract spherical polyhedron} can be constructed from a spherical drawing 
of a $3$-connected planar graph $G$ (no edges crossings), adding the regions created in the drawing as the `faces' of
the polyhedron.  
This face structure is unique, given $3$-connectivity.   
 
\begin{definition}
{\rm A} block and hole polyhedron {\rm $\poly$ with vertex set $V,$ edge set $E,$ and face set $\mathcal{F}$ is an abstract spherical
polyhedron whose faces $\mathcal{F} = (\B_{\poly}, \Ho_{\poly}, \Su_{\poly})$ are partitioned into three mutually disjoint sets, 
$\B_{\poly}, \Ho_{\poly}$, and $ \Su_{\poly}$. 
The set $\B_{\poly}$ contains the faces designated as blocks and the set
$\Ho_{\poly}$ contains the faces designated as holes. The remaining faces are triangulated on their vertices, 
and the collection of resulting triangular faces forms the set $\Su_{\poly}$.}
\end{definition}
\begin{figure}
\begin{center}
 \subfloat[]{\label{fig:polyhedraA}\includegraphics[width=2in]{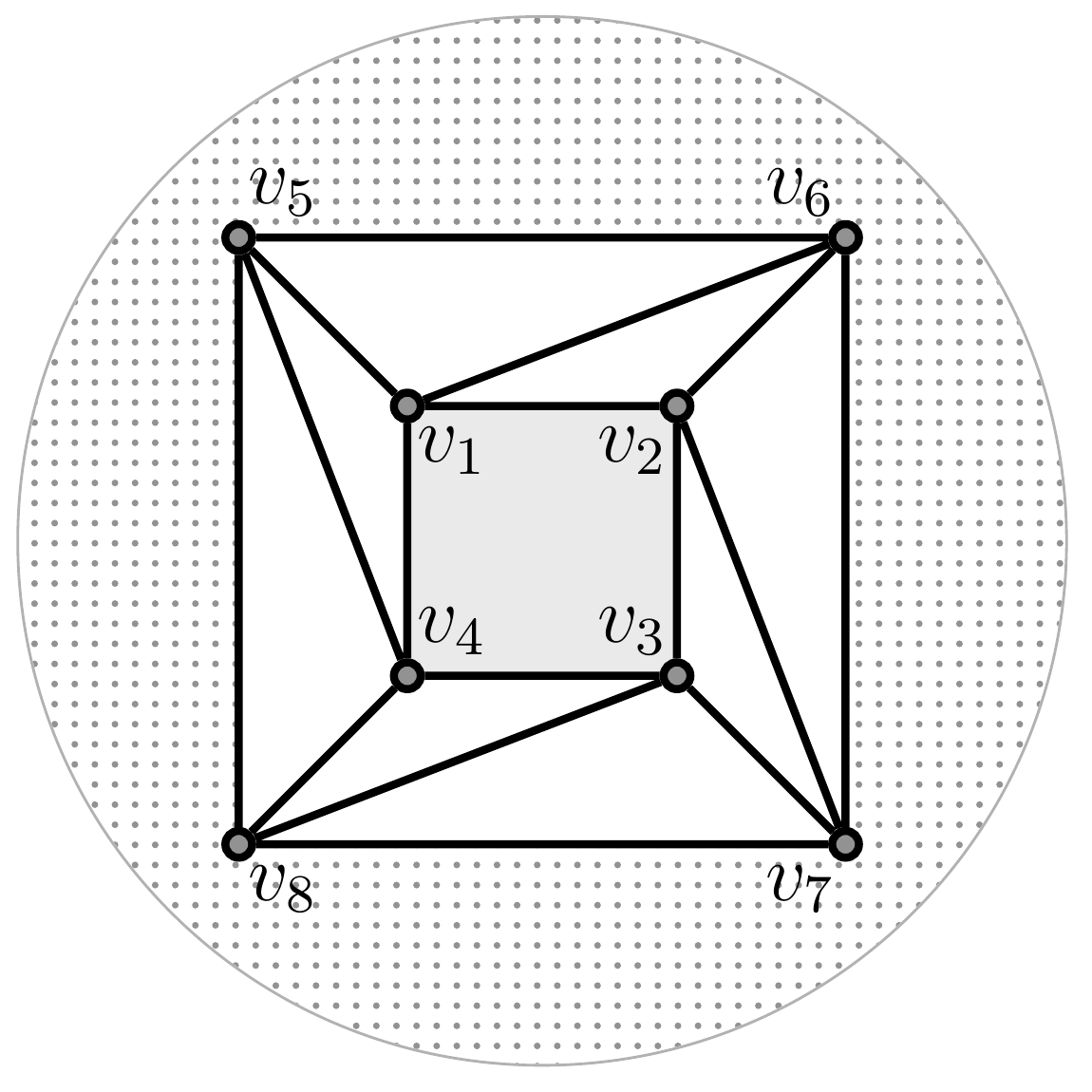}}   \hspace{0.2in}
 \subfloat[]{\label{fig:polyhedraB}\includegraphics[width=2in]{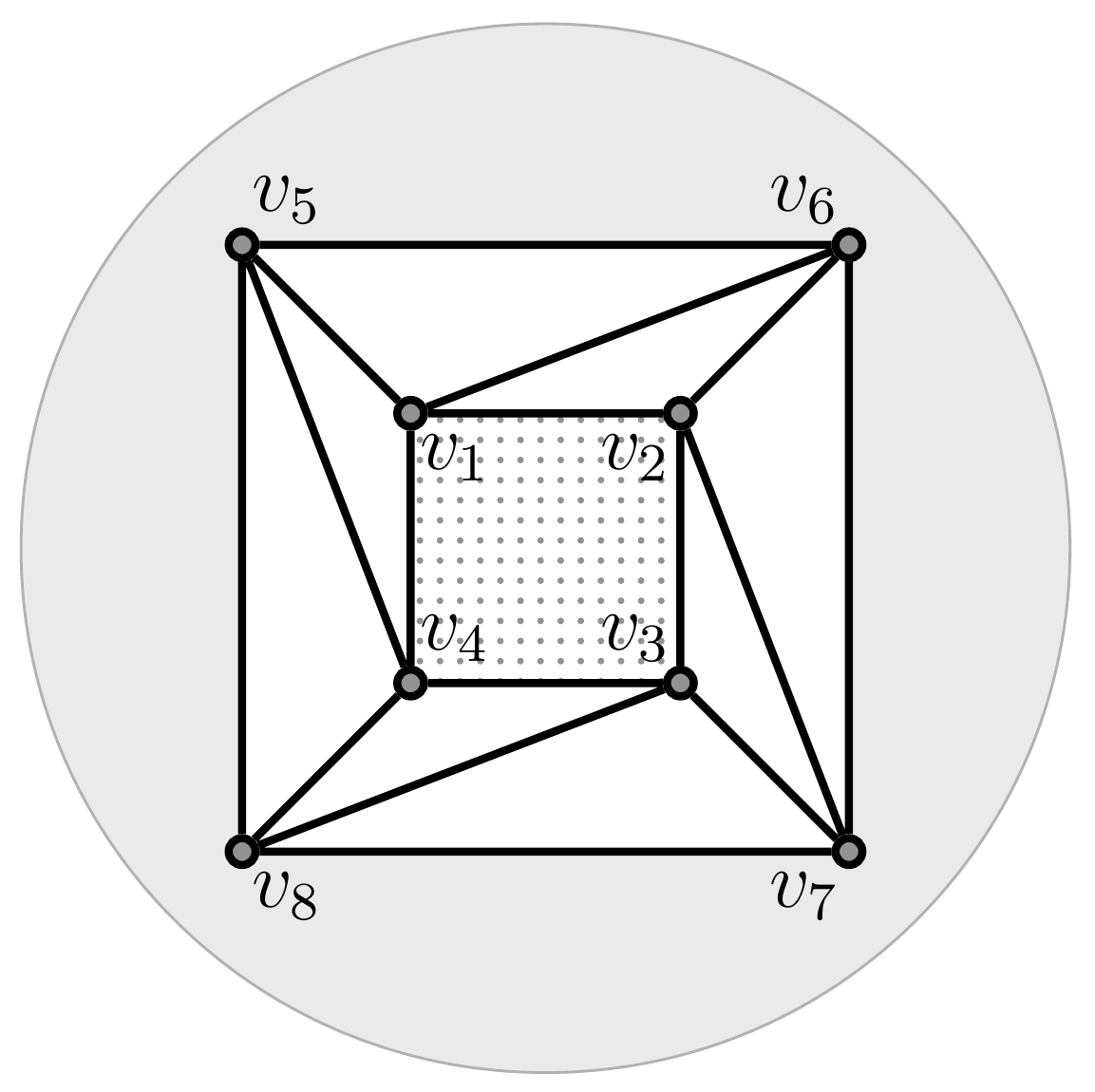}}   
\caption{An example of a block and hole polyhedron $\poly$ shown in (a). Blocks are denoted by grey faces, and holes are denoted by dotted faces. The circular region is understood to be a connected face, in this case a hole face of the polyhedron. The swapped version of this polyhedron, $\spoly$ is shown in (b). \label{fig:polyhedra} }
\end{center}
\end{figure}

In Section \ref{sec:moreGeneralSpheres}, we consider the small modifications needed to include $2$-connected planar graphs.

Central to our analysis of block and hole polyhedra is the process of `swapping', in which holes become
blocks and blocks become holes (See Figure \ref{fig:polyhedra}).

\begin{definition}
{\rm
Given a block and hole polyhedron $\poly$, the} swapped block and hole polyhedron {\rm $\spoly =  (\B_{\spoly},
\Ho_{\spoly}, \Su_{\spoly})$ is the block and hole polyhedron with blocks and holes interchanged; that is, $\B_{\spoly}
= \Ho_{\poly}$, $\Ho_{\spoly} = \B_{\poly}$ and $\Su_{\spoly} = \Su_{\poly}$.}
\end{definition}

It is immediate that $\ov{\spoly} = \poly$.

Recall that each edge $\{i, j\}$ of a block and hole polyhedra $\poly$ joins two distinct vertices $i$ and $j$ of $\poly$, 
and separates two distinct faces, $F^k$ and $F^m$ of $\poly$ (Figure \ref{fig:edgepatch}).
\begin{figure}
\begin{center}
\includegraphics[width=5in]{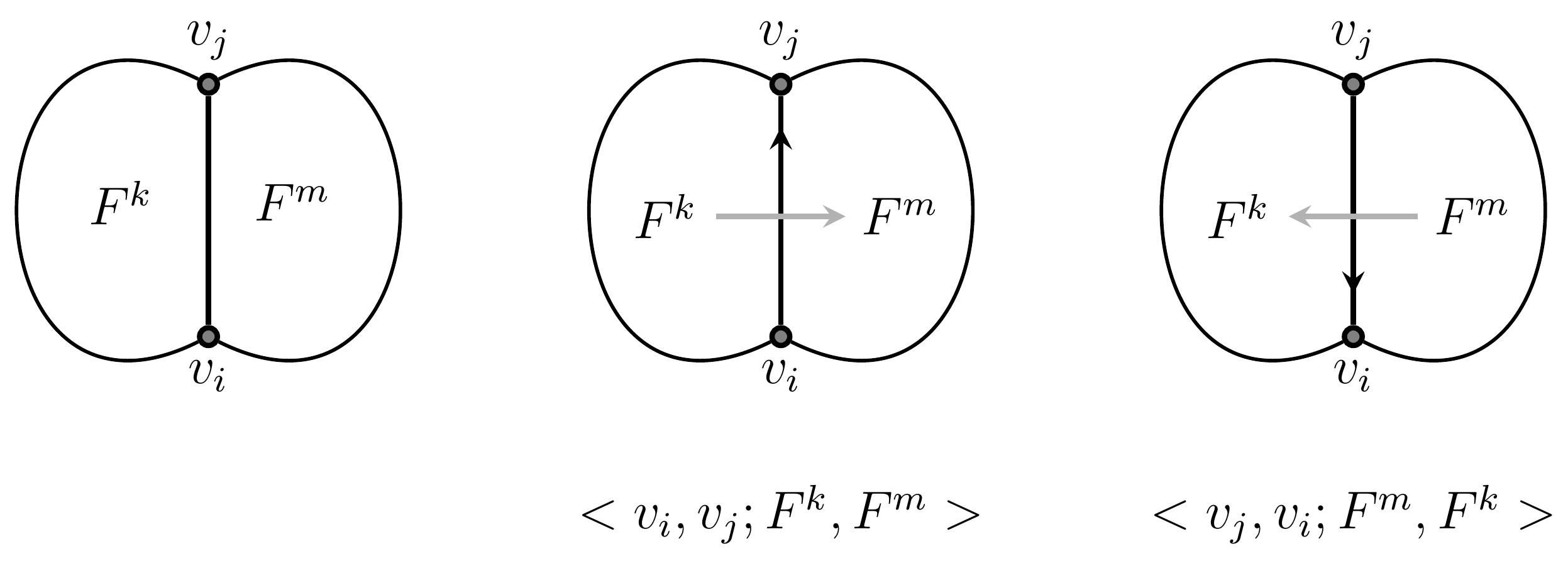} 
\caption{Two vertices joined by an edge and the two faces meeting at an edge form two
oriented edge patches.\label{fig:edgepatch}}
\end{center}
\end{figure}
We introduce a pairing of the vertex-edge-vertex pairs and the face-edge-face pairs obtained from 
the orientation of the spherical surface, $\poly$. 
We write $\ijkm$ for the {\it edge patch}  
with oriented edge from $i$ to $j$ and the oriented pair of faces $(F^k, F^m)$
which crosses the oriented edge after a $90^{\circ}$ counterclockwise turn.  The {\it reversed patch} is
$\jimk$; see Figure \ref{fig:edgepatch}.  
This spherical drawing of the block and hole polyhedra $\poly$ also gives a cycle of faces and edges around each vertex.  
We use such cycles and patches in our later analysis.

%%%%%%%%%%%%%%%%%%%%%%%%%%%%%%%%%%%%%%%%%%
\subsection{Stresses of Frameworks} \label{sec:stresses}
%%%%%%%%%  First draft:  Wendy July 10, 
%%%%%%%%%  Some revision Walter July 27 
%%%%%%%%%  Revised Wendy Aug 17 
%%%%%%%%%  Revised Elissa Aug 25
%%%%%%%%%  Wendy Dec 15
%%%%%%%%%  Example Elissa Jan 25
%%%%%%%%% Walter October 12
%%%%%%%%%%%%%%%%%%%%%%%%%%%%%%%%%%%%%%%%%%

In this section, we outline the basic projective theory of static rigidity for frameworks, following \cite{crapowhiteley}. A {\it bar and joint framework},
${\mathcal F} = (G, \p)$ in projective 3-space is a finite graph ${G} = (V, E)$, where $V = \{1, 2, \ldots, v\}$ is a
set of {\it vertices} and $E$ is a set of unordered pairs of vertices called {\it edges}, 
and a mapping $\p: V \mapsto {\PS}^3$ such that if $\{i, j\} \in E$ then $\p(i) \neq \p(j)$. If $i \in V$, we call
$\p(i) = \p_i$ a {\it joint} of ${\mathcal F}$.  
We call $(\p_{i}, \p_{j})$ a {\it bar} of the framework, and when the framework is clear we denote the bar as $(i, j)$ or
$ij$.   
We refer to the mapping $\p$ as a {\it configuration}, to $E$ as the set of bars, and to $V$ as the set of joints of the
framework. 

We are particularly interested in internal stresses that occur on the bars and joints of frameworks. For ease of
notation and simplicity, we do this in a projective setting. For an introduction to Grassmann-Cayley
algebra and projective geometry, see 
\cite{crapowhiteley,whitehandbook,whiteGrassmann}. 
   
Let $\as = (a_1, a_2, a_3, 1)$ and $\bs = (b_1, b_2, b_3, 1)$ be two points in $\PS^3$.    
The line through $\as$ and $\bs$ may be represented by the {\it Grassmann} or 
{\it exterior product}, $\as \vee \bs = \as\bs$ (read a join b),
defined to be the six 2x2 minors
$(d_{14}, d_{24}, d_{34}, d_{23}, d_{31}, d_{12})$ of the 2x4 matrix
$$\left[ \begin{array}  {cccc}
a_1 & a_2 & a_3 & 1 \\
b_1& b_2 & b_3 & 1
\end{array} \right]. $$  
\noindent The exterior product $\as\bs$ is known as both a {\it 2-extensor} and the 
{\it Pl\"ucker coordinates} of the line through $\as$ and $\bs$. 

To understand the rigidity of such a framework, we examine the effect of applying forces (loads) to the structure.  
A force can be thought of as a directed segment of the line though the projective points
${\bf f} = (f_1, f_2, f_3, 0)$ and $\as = (a_1, a_2, a_3, 1)$, and hence the load applied to $\as$ is represented by the
2-extensor $\bf{F} = \bf{f}\as$.

We assume all forces are applied at the joints of the framework.  
An {\it external load} on a framework is an assignment of loads $\Lb = (\Lb_1, \Lb_2, \ldots, \Lb_v)$ to the joints $V =
(\p_1, \p_2, \ldots, \p_v)$ of the framework such that each $\Lb_i$ passes through the joint $\p_i$. The external load is said to be in
{\it equilibrium} if 
\begin{equation}\sum_{i=1}^v \Lb_i = \0\label{eqn:EquilibruimLoads}
\end{equation}
This set of six equations is independent if the points are not collinear, giving a vector space of equilibrium loads of  dimension $3|V|-6$.  

 A {\it resolution} of the equilibrium load by a framework is an assignment
of scalars $\lambda_{ij}$ to the bars of the framework such that for each joint $\p_i \in V$,
\begin{equation}
\Lb_i + \sum_{\{j|(i,j)\in E\}} \lambda_{ij}\p_i\p_j = \0.
\label{eqn:resolutionOfEquilibruimLoads}
\end{equation}

\begin{definition}
{\rm A framework ${\mathcal F} = (G, \p)$ is  }statically rigid {\rm if
every equilibrium load of the framework has a resolution by the bars of the framework. }
\end{definition}
 
The set of equations given in the vector equation 
(\ref{eqn:resolutionOfEquilibruimLoads}) defines a linear transformation from the vector space of
resolutions (by the bars of the framework) to the vector space of equilibrium loads. 

A framework is called {\it isostatic} if its bars 
form a basis for the space of equilibrium loads,
in other words, the framework is called isostatic if it is minimally statically rigid.
It follows that in an isostatic framework, $|E| = 3|V| - 6$ (unless $|V|\leq 2$). 

Throughout the remainder of this paper, we will focus on internal forces acting within the entire framework ${\mathcal F} = (G, \p)$.  We
consider the possible tensions and compressions within the bars of a framework.  
Specifically, we have the {\it bar load }$\Bb_{ij}$  which applies equal and opposite forces at the two ends of a bar:  $\p_i \p_j $ at $p_{i}$ and 
$ \p_j \p_i$ at $p_{j}$ and $0$ at all other vertices.  Together, these two form an equilibrium load, 
and any multiple $\lambda_{ij}\Bb_{ij}$ is a tension ($\lambda_{ij} <0$) or compression ($\lambda_{ij} >0$) in the bar.  
The subspace generated by these internal bar forces  is the space $\mathcal{B}$ of resolved loads.  

We say these internal forces are in {\it equilibrium at each joint $\p_j$} when $\sum_{j} \lambda_{i,j} \p_i \p_j  = \0$ for every joint $\p_i$ such that $(i, j) \in E$.   
A {\it self-stress} on a bar and joint framework ${\mathcal F}$ is an assignment of scalars $\lambda_{ij}$ 
to the bars $(i,j)$ of the framework such that 
for every vertex $i$,
\begin{equation}
\sum_{\{j|(i,j) \in E\}} \lambda_{ij}(\p_i\p_j) = \0.
\label{eqn:selfStressCondition}
\end{equation}
A self-stress is called {\it non-trivial} if some $\lambda_{ij} \neq 0$. For simplicity, we will call a non-trivial
self-stress a {\it stress}.  A framework is called {\it independent} if it has only a trivial self-stress; otherwise, a
framework has a non-trivial self-stress and is called {\it dependent}.

In the language of the bar loads, a self-stress is a linear dependence among the bar loads.   A framework is statically rigid if $\mathcal{B}$ is the entire space of 
equilibrium loads.  A framework is isostatic if $\mathcal{B}$ is a basis for the space of equilibrium loads. 

A {\it cut set} $E'$ of a framework is a subset of the edges of the framework (in other words, $E' \subset E$) whose
removal separates the framework into two or more distinct components. The following theorem gives a useful property of such cut sets.

%\begin{theorem}{\rm{(\cite{crapowhiteley})}
%Let $E'$ be the cut set of a framework and let $C$ be a component of the framework with the edges $E'$ removed. Let the
%set of joints in $C$ be the set $\{\p_k, \p_{k+1}, \ldots, \p_m\}$,  and let $(\p_i, {\bf{q}}_i) \in E'$.  Then for any stress $\Lambda$ on the
%framework,
%$$ \sum_{i = k}^{m} \lambda_i \p_i \bf{q}_i = \0, $$}
%where $\lambda_i \in \Lambda$.
%\label{cutsetTheorem}
%\end{theorem}

\begin{theorem}{\rm{(\cite{crapowhiteley})}
Let $E' = \{(\p_i, \q_i)|  k \leq i \leq n \}$ be the cut set of a framework ${\mathcal F} = (G, \p)$.
%and let $C$ be a component of the framework with the edges $E'$ removed. 
% Let the set of joints in $C$ be the set $\{\p_k, \p_{k+1}, \ldots, \p_n\}$, with $k \leq m$.  
Then for any stress $\Lambda$ on the framework,
$$ \sum_{i = k}^{n} \lambda_i \p_i \q_i = \0, $$}
where $\lambda_i \in \Lambda$.
\label{thm:cutset}
\end{theorem}

Theorem~\ref{thm:cutset} says that the forces from a self-stress acting across the cut-set onto a component $C$ with edges from the cut set removed, 
are an equilibrium load onto this component.   
The overall equilibrium conditions for the larger self-stress guarantee that the coefficients of the self-stress within the 
component $C$ are a resolution of this cut-set load.  

In our later work, we use the following  {\it isostatic substitution principle}, which says, roughly speaking, 
that within a given framework, we can substitute an isostatic subframework attached on a subset of vertices $V'$ for another isostatic subframework attached at $V'$, 
without changing the static rigidity of the overall framework.  
%In our later work, we use the following  {\it isostatic substitution principle}, which says, roughly speaking, 
%that we can interchange isostatic subframeworks on subsets of the vertex set 
%without changing the static rigidity of the overall framework.  
%We will need a general {\it isostatic substitution principle} in our later work.  
This will follow from general principles about bases, spanning sets, and linear dependencies in vector spaces, translated into the language of  
the particular spaces of interest in statics:  (i) the isostatic frameworks, whose bars are bases for the equilibrium loads; 
(ii) the self-stresses which are linear dependencies among the bars; and (iii) Theorem~\ref{thm:cutset} which converts a self-stress across a cut set into an 
equilibrium load on the component.   A special varient of this principle was used in \cite{infp1}.   

We begin with the simple case where no vertices are added.

\begin{lemma}[Isostatic Substitution Principle]  \label{lem:isostaticSubstitution}  
Given a geometric framework ${\mathcal F} = (G,\p) = ((V,E), \p)$, 
with an isostatic subframework 
${\mathcal F}'= (G',\p')$  on a set of vertices $V' \subset V$, 
then replacing  ${\mathcal F}'= (G',\p|_{V'})$ with another isostatic subframework ${\mathcal F}^{''}= ((V',E''),\p|_{V'})$ on the same vertices  $V'$,  gives a new framework   ${\mathcal F^{*}} = (G^{*}, \p) = ((V,E^{*}), \p) $ which has the same space of resolvable loads $\mathcal{B}$, and an isomorphic space of self-stresses.  
\end{lemma}

\begin{proof}  The core idea is that, working within a vector space of equilibrium loads, isostatic subframeworks are bases for the subspace  
of possible equilibrium loads on the vertices of the subframework.  If ${\mathcal F}'$ and 
${\mathcal F}''$ use the same vertices $V'$,  then it is clear 
that we are replacing one basis for the space of equilibrium loads on $\p|_{V'}$ with another basis for this same subspace of equilibrium loads.  

The larger vector space of resolved equilibrium loads on the entire framework $\mathcal{B}$ has the same 
dimension as $\mathcal{B^{*}}$.  We have the same number of bars (basis are all the same size in a vector space) so we have with equivalent spaces of self-stress 
(dependencies of the bar loads).   
\qed  \end{proof}

We offer an extension of this, in which the substituted subframework may possess more vertices than the attaching set of vertices $V'$. That is, the substituted framework has vertices $V' \cup U$, where $U\cap V = \emptyset$. In particular, if $V'$ is a set of block vertices, then an isostatic framework on these vertices may be substituted with another isostatic framework on the vertices $V' \cup U$, while maintaining an isomorphic space of self-stresses.  

\begin{corollary}[General Isostatic Substitution Principle]  \label{cor:GeneralisostaticSubstitution}  
Given a geometric framework ${\mathcal F} = (G,\p) = ((V,E), \p)$, 
with an isostatic subframework 
${\mathcal F}'= (G',\p')$ on a set of vertices $V' \subset V$, 
then replacing  ${\mathcal F}'= (G',\p|_{V'})$ with another isostatic subframework 
${\mathcal F}^{''}= ((V'\cup U,E''),\p|_{V'}\cup \q$) on the same vertices  $V'$, plus possible additional vertices $U$ in general position $\q$ relative to $\p|_{V'}$,  gives a new framework   ${\mathcal F^{*}} = (G^{*}, \p) = ((V\cup U,E^{*}), \p\cup \q) $ which has an isomorphic space of self-stresses.  
\end{corollary}

\begin{proof}
If we have  added vertices in  $U$ which are not in $V'$, then we can add these vertices to the original framework, at positions $\q$ 
by inserting general position $3$-valent
vertices (non-coplanar) connected within ${\mathcal F}'$ to get a modified framework and subframework with an isomorphic space of  self-stresses \cite{taywhiteley}. 
Having done that to each of the pieces to be compared, we now have returned to 
the case of the previous lemma, and we have, overall, isomorphic spaces of unresolved equilibrium loads and of self-stresses.  
This will have a larger space of resolvable loads -- it increased the dimension by $3|U|$. 
\qed 
\end{proof}

This could be extended one more level - by having some vertices $W$ which are in ${\mathcal F}'$ but not attached to the rest of  ${\mathcal F}$, which are dropped, while new vertices $U$ are added.  This level of generality is not needed here.  

%%%%%%%%%%%%%%%%%%%%%%%%%%%
\subsection{Motions of Body and Hinge Structures} \label{sec:motions}
%%%%%%%First Draft Elissa August 25%%%%
%%%%%%%Scrapped and rewritten Elissa November 6%%%%%
%%%%%%   Example Elissa Jan 25

It is traditional to analyze the rigidity of a framework with both the tools of static rigidity, as we have here, and 
tools of infinitesimal rigidity \cite{graver,wchapter,handbook}.   We will not present this kinematic theory for bar and joint frameworks here, as it is not needed.  
   
 An alternative approach  to kinematics rigidity used in this paper, is the theory of infinitesimal motions for bodies and hinges.   
At a basic level, the motion of a framework or structure can be  presented by velocity vectors assigned at all points of the 
 structure in such a way as to not distort the geometric shape of the structure. 
 That is, the velocity assignments must preserve the pairwise distances between points in the same rigid body of that structure. 
 If the only such assignments correspond to rigid motions of the whole structure, 
 then we say that the object is infinitesimally rigid. 

In this section we focus on 
the structures of bodies and hinges we will use in the rest of this paper, using a projective algebra.  A more detailed 
exposition, with examples and motivation, can be found in \cite{crapowhiteley}, and  
the Euclidean basis for this theory is 
described in \cite{graver}. 

%%% Body and Hinge
A {\it body and hinge} structure $G^*(\Db)$ in $\R^3$ is a graph 
$G^* = (\mathcal{F}, \mathcal{D})$ 
together with a mapping $\Db$ from $\mathcal{D}$ into the $2$-extensors of $\PS^3$. 
We think of the vertices, $F^{k} \in \mathcal{F}$ of $G^*$ as representing rigid bodies, and the edges $e \in \mathcal{D}$ represent hinges. 
If the vertices $F^k \in \mathcal{F}$ and $F^m \in \mathcal{F}$ are joined by an edge $e \in \mathcal{D}$, then we write $e = D^{km}$. 
A body and hinge structure is {\it connected} if its underlying graph $G^*$ is connected. 

%Recall that a 2-extensor ${\bf y}$ represents a line given in Pl\"{u}cker coordinates, and can be written as 
%${\bf y} = \as \vee \bs $ for some choice of $\as$ and $\bs$ in $\PS^3$.  If ${\bf y} = (y_1, \dots, y_6) \in \R^6$, then 
%$$y_1 y_4 + y_2 y_5 + y_3 y_6 = 0.$$ 

%%% Screw Centres
A {\it screw centre of motion for a body} (a {\it screw}) is any vector ${\bf S} \in \R^6$. 
This $6$-vector represents a motion of $\PS^3$, in that it encodes all the information for defining the velocities  of any combination of rotations and translations.
We can also interpret $\bf{S}$ as a weighted sum of $2$-extensors in $\PS^3$. 
%By Poinsot's central axis theorem (\cite{crapowhiteley}), an arbitrary screw can be written as the sum of at most two $2$-extensors, 
%one representing a rotation about an axis, 
%and the other a translation along the same axis. 
% An arbitrary screw, $\bf{S} \in \R^6$ may itself be a single rotation or translation. 
%If this is the case, the coordinates of $\bf{S}$ will satisfy $y_1 y_4 + y_2 y_5 + y_3 y_6 = 0$ and
%$\bf{S}$ will represent a line in $\PS^3$. In most cases, however, the coordinates of $\bf{S}$ will not satisfy the relation above,
%and we can interpret $\bf{S}$ as a weighted sum of $2$-extensors in $\PS^3$.

We note that the edge $D^{km}$ of $G^*$ maps to a $2$-extensor $\Db^{km}$ of $G^*(\Db)$, 
and we write $\Db^{km} = {\bf a} \vee {\bf b}$, for some ${\bf a}, {\bf b} \in \PS^3$. 
Recall that ${\bf a} \vee {\bf b} = - ({\bf b} \vee {\bf a})$, and $\Db^{mk} = {\bf b} \vee {\bf a}$, thus $\Db^{km} = - \Db^{mk}$. 

%%% Infinitesimal Motions
An {\it infinitesimal motion} of a body and hinge structure is an assignment of a screw centre $\SB^i$ to each 
vertex $F^{k}$ of the graph $G^*$, such that for each edge $D^{km} \in \Db$, 
$$\SB^k - \SB^m = \omega^{km} D^{km}, \ \ \text{for some scalar}\  \omega^{km} \in \R.$$
We think of this equation as an assignment of a center for the rigid motion of each body of the structure, with 
the constraint that points along the hinge receive the same velocity from each of the adjacent centers. 

\begin{definition}
{\rm A body and hinge structure is} infinitesimally rigid {\rm if every infinitesimal motion is a rigid motion.} 
\end{definition}
That is, it is infinitesimally rigid if every infinitesimal motion is trivial, with all vertices (bodies) 
receiving the same screw center assignment, $\SB$. In this case the whole structure moves 
according to the motion encoded by $\SB$. 

%%% Motion Assignments
There is a more condensed way of describing the motion of a given structure. 
A {\it motion assignment} is an assignment of scalars $\omega^{km}$ to the hinges $\Db^{km}$ of 
the body and hinge structure such that: 
\begin{enumerate}
\item $\omega^{km} = \omega^{mk}$, and 
\item $\sum \omega^{km} \Db^{km} = \0$ for every closed cycle of panels and hinges in the structure. 
\end{enumerate}

%%% Equivalence of Motion Assignments and Infinitesimal Motions
The following relationship exists between the infinitesimal motions and the motion assignments of a body and hinge structure:
\begin{proposition} {\rm \cite{crapowhiteley} }For a connected body and hinge structure $G^*(\Db)$ with a 
designated body $F^* \in \mathcal{F}$, there is a one-to-one correspondence between the infinitesimal motions of 
the structure with $\SB^{F^*} = \0$ and the motion assignments on the structure. A motion assignment represents a non-trivial 
motion if, and only if $\omega^{km} \neq 0 $ for some hinge $D^{km}$. 
\end{proposition} 

\noindent{\bf Remark 1.} Note that we can change the body we consider fixed in the above. 
That is, given an infinitesimal motion of the structure with $\SB_1 = \0$ and $\SB_m \neq \0$, 
we can convert this to an equivalent motion with ${\SB^m}' = \0$ by setting ${\SB_k}' =\SB_k - \SB_m$. 
This new motion will generate the same scalars for the motion assignment, 
but we have a new frame of reference for the space.  In the rest of this paper we will focus on the 
motion assignment as the record of the motion. 
\\

\medskip

\noindent{\bf Remark 2.} We did not introduce the full vocabulary and notation for the projective
analysis of infinitesimal motions of bar and joint frameworks in this section.  
In particular, we have not confirmed directly that the space of infinitesimal motions of a bar and joint framework is unchanged by isostatic substitution. 
Our static analysis, and associated vector spaces does include a space which is isomorphic to the space of non-trivial infinitesimal motions.  
Namely, if we take the orthogonal complement of the space of resolved loads $\mathcal{B}^{\perp}$ within the space of all equilibrium loads, 
we get a space 
$\mathcal{U}= \mathcal{B}^{\perp}$.  
This vector space has the same dimension as the space of non-trivial infinitesimal motions, as is verified in a more complete presentation 
\cite{infp1, handbook, wchapter}.  

We will not use any details of this proxy space below, but it is relevant to recognizing that the Isostatic Substitution Principle, 
Lemma~\ref{lem:isostaticSubstitution}, also shows that the substitution gives an isomorphism of the spaces of unresolved loads, 
or equivalently, an isomorphism of the spaces of non-trivial infinitesimal motions.  
The the proof of Corollary~\ref{cor:GeneralisostaticSubstitution} also shows that such a substitution produces an an isomorphism of the 
spaces of unresolved loads, or equivalently, an isomorphism of the spaces of non-trivial infinitesimal motions.

Of course, our body-hinge structures do not depend on what is used to build the body, as long as each body in the structure is statically 
(infinitesimally) rigid.

%%%%%%%%%%%%%%%%%%%%
\subsection{Notation and Connections on Block and Hole Polyhedra} \label{sec:notationAndConnections}
%%%%%%%%%Drafted Elissa August 25, 
%%%%%%%%% revised Wendy Jan 25       
%%%%%%%%%%  %  Revised by Walter February 16, October 02 2009.   
%%%%%%%%%%%%%%%%%%%%

Informally, when we study a block and hole polyhedron as a bar and joint framework at $\p$ we want the block to be
(statically) rigid.  For each block $B_{\poly}^{k}$ we introduce some added subframework
$I^{k}=(V^{k},U^{k}, E^{k})$ where $V^{k}$ is the set of vertices of the block face, $U^{k}$ are some added
new vertices (possibly empty), and $E^{k}$ are some added edges on $V^{k}\cup U^{k}$ such that, with the 
edges of the polygon added, $I^{k}$ becomes an isostatic bar and joint framework realized at $\p|_{V^{k}}$
with generic positions for $U^{k}$. 

We note that if all vertices of a block face are collinear, then there is no possible $I^{k}$ 
making an isostatic framework, 
since this collinear polygon is already dependent.  
If the vertices of a block face boundary are coplanar, then we will need to add some vertex in $U^{k}$ 
off the plane in order 
to achieve an isostatic subframework.  If the vertices of the block face boundary polygon are not coplanar,
then we can create $I^{k}$ with only added edges and no added vertices.  

The key properties of the block and hole frameworks do not depend on which isostatic subframework is 
inserted for each block, provided that the boundary polygon of the original face is used as part of the
framework.  This is captured by the Isostatic Substitution Principle given in 
Lemma \ref{lem:isostaticSubstitution}.

With this in mind, in the remainder of this paper we will be non-specific about 
which isostatic subframework is used in place of a block, with the
exception that we do assume that the original polygon of the face is present among the edges of the
isostatic framework. 

\begin{definition} {\rm The} static framework graph, {\rm $G_{S}(\P)$, is the graph of a block and hole 
polyhedral framework with the added
frameworks $I^{k}$ for each block.  Since we do not 
pay attention to the isostatic subframeworks on the blocks, we consider $G_{S}(\P)$ 
to be a representative framework 
among an equivalence class of graphs (with various isostatic blocks inserted into the
block faces).  }
\end{definition} 

\begin{figure}
\begin{center}
\includegraphics[width=6in]{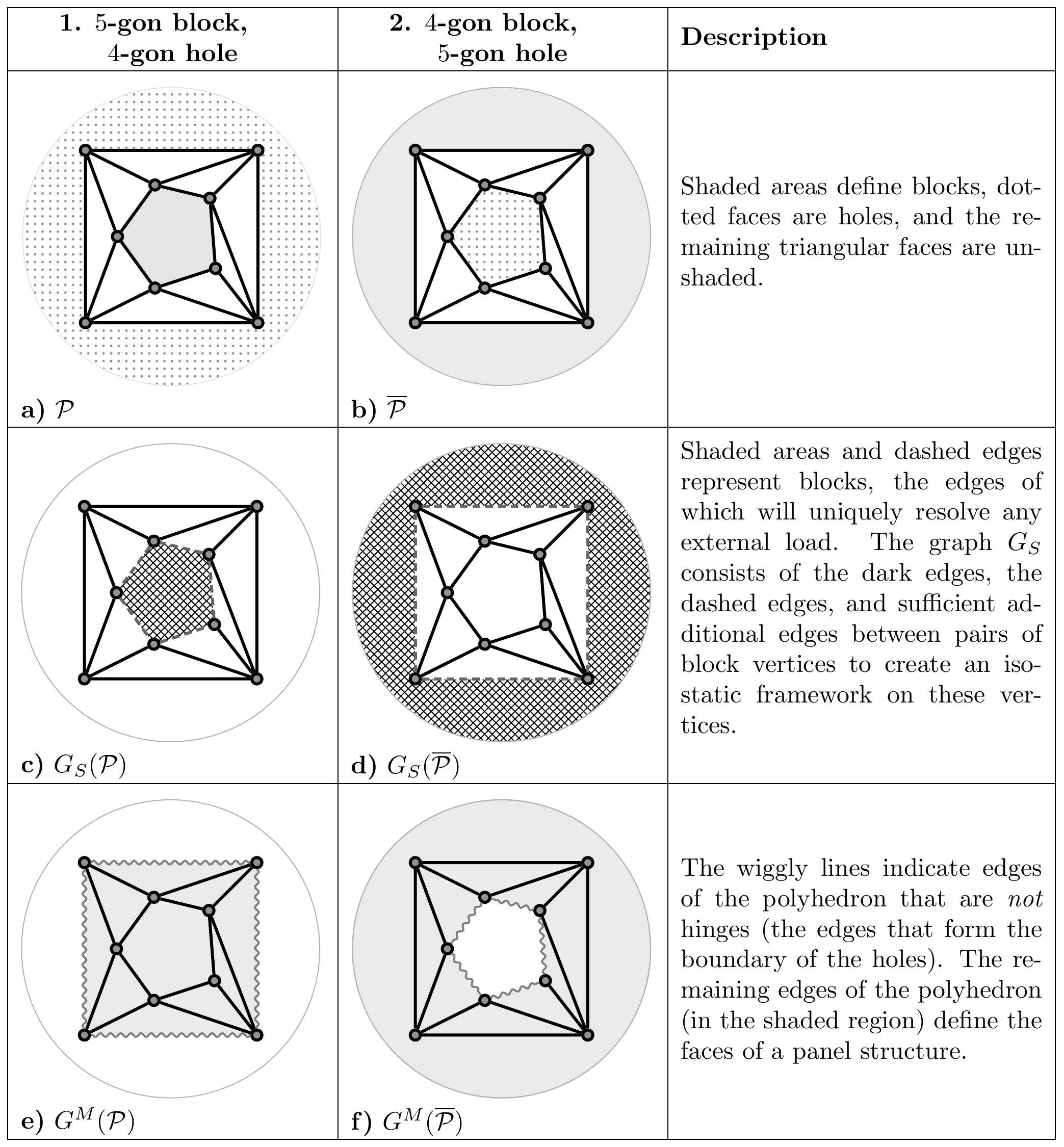}
\caption{Examples of block and hole polyhedra, and their associated graphs $G_S(\poly)$ and $G^M(\poly)$. \label{fig:example} }
\end{center}
\end{figure}

This graph $G_{S}(\P)$ will be used to track the stresses of such frameworks on $\poly$. The graphs $G_{S}(\P)$ and $G_{S}(\ov\P)$
exist for every block and hole
polyhedron (see Figure~\ref{fig:example}, (c) and (d)). At a configuration $\p$ they form  bar and joint frameworks with
well defined spaces 
of self-stresses, which we denote ${\mathcal{S}}(G_{S}(\poly,\p))$, ${\mathcal{S}}(G_{S}(\ov\poly,\p))$. 
The space of residual unresolved equilibrium loads for these frameworks (our proxy space
for the bar and joint infinitesimal motions), is denoted by 
${\mathcal{M}}(G_{S}(\poly,\p)),{ \mathcal{M}}(G_{S}(\ov\poly,\p)) $.

At the  heart of this paper, we 
explore the connections (essentially an isomorphism for appropriate polyhedra)
 between the space ${\mathcal{M}}(G_{S}(\poly,\p))$ for the original polyhedron and the stress space
${\mathcal{S}}(G_{S}(\ov\poly,\p))$ for the swapped polyhedron.

As an intermediary in the proofs, we will use  
an induced body and hinge structure on $(\poly,\p)$ in place of  ${\mathcal{M}}(G_{S}(\poly,\p))$  
to track these connections.   This body and hinge structure is composed of the rigid bodies 
(surface faces and bodies, but not holes), and edges between rigid faces of
the underlying spherical block and hole polyhedron $\poly$ to form the body and hinge 
 polyhedron $G^M(\poly)$ (Figure~\ref{fig:example}, (e) and (f)).  For a particular configuration $\p$, we denote the vector space of motion assignments on this structure by $\mathcal M(G^M(\poly, \p))$. As we will see, for block and hole polyhedra $\poly$ satisfying certain conditions, the spaces $\mathcal M(G^M(\poly, \p))$ and 
${\mathcal{M}}(G_{S}(\poly,\p))$ are isomorphic. We first address the situations in which these spaces are not isomorphic.

The structure $G^M(\poly,\p)$ may not form a sufficiently connected graph for the body and hinge structure to 
have only the motions of the underlying bar and joint framework, which make up the vector space ${\mathcal{M}}(G_{S}(\poly,\p))$.  
Figure \ref{fig:gussetNeed} (a) and (b) illustrate this with a polyhedron $\poly$ for which the body-hinge structure $G^M(\poly)$ is disconnected.  Figure \ref{fig:gussetNeed} (c) and (d) depict a related example where the connectivity is sufficient to capture the motions of the underlying bar and joint framework.
This problem of connectivity arises only when
$\poly$ has a vertex with two or more holes at the vertex, as in figure \ref{fig:gussetNeed} (a) and (b). 
The graph of rigid faces and shared edges $G^M(\poly)$, no longer capture the nature of the motions of the rigid faces at this vertex.   As such, the creation of the graph $G^M(\poly)$ may introduce extra motions not in the underlying framework.  This observation motivates the following definition:

\begin{figure}
\begin{center}
\subfloat[$\poly$]{\label{fig:gussetNeedPoly}\includegraphics[width=2in]{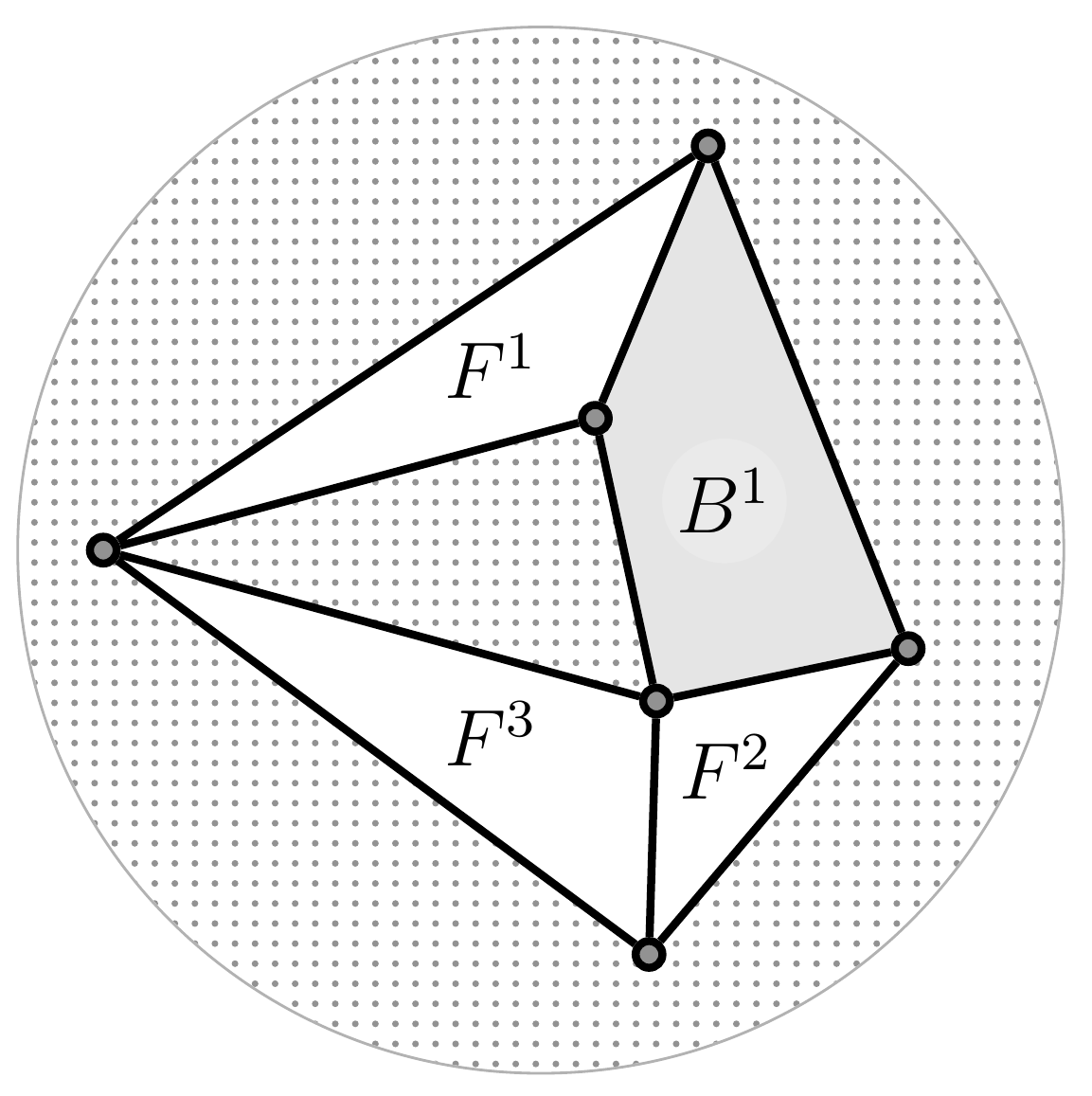}} \hspace{0.2in}               
\subfloat[$G^M(\poly)$]{\label{fig:gussetNeedGM}\includegraphics[width=2in]{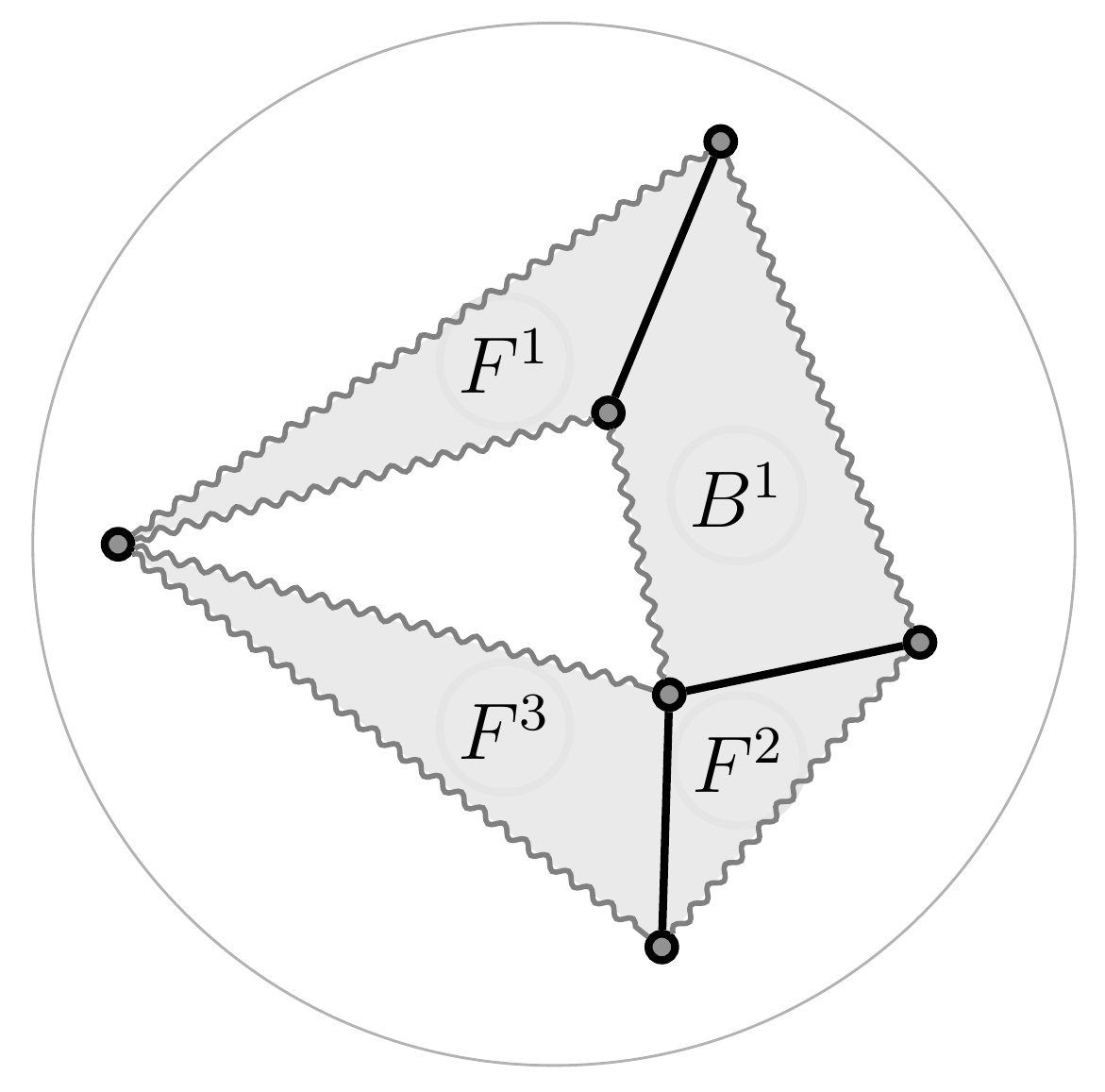}}\\
\subfloat[$\poly'$]{\label{fig:gussetedPoly}\includegraphics[width=2in]{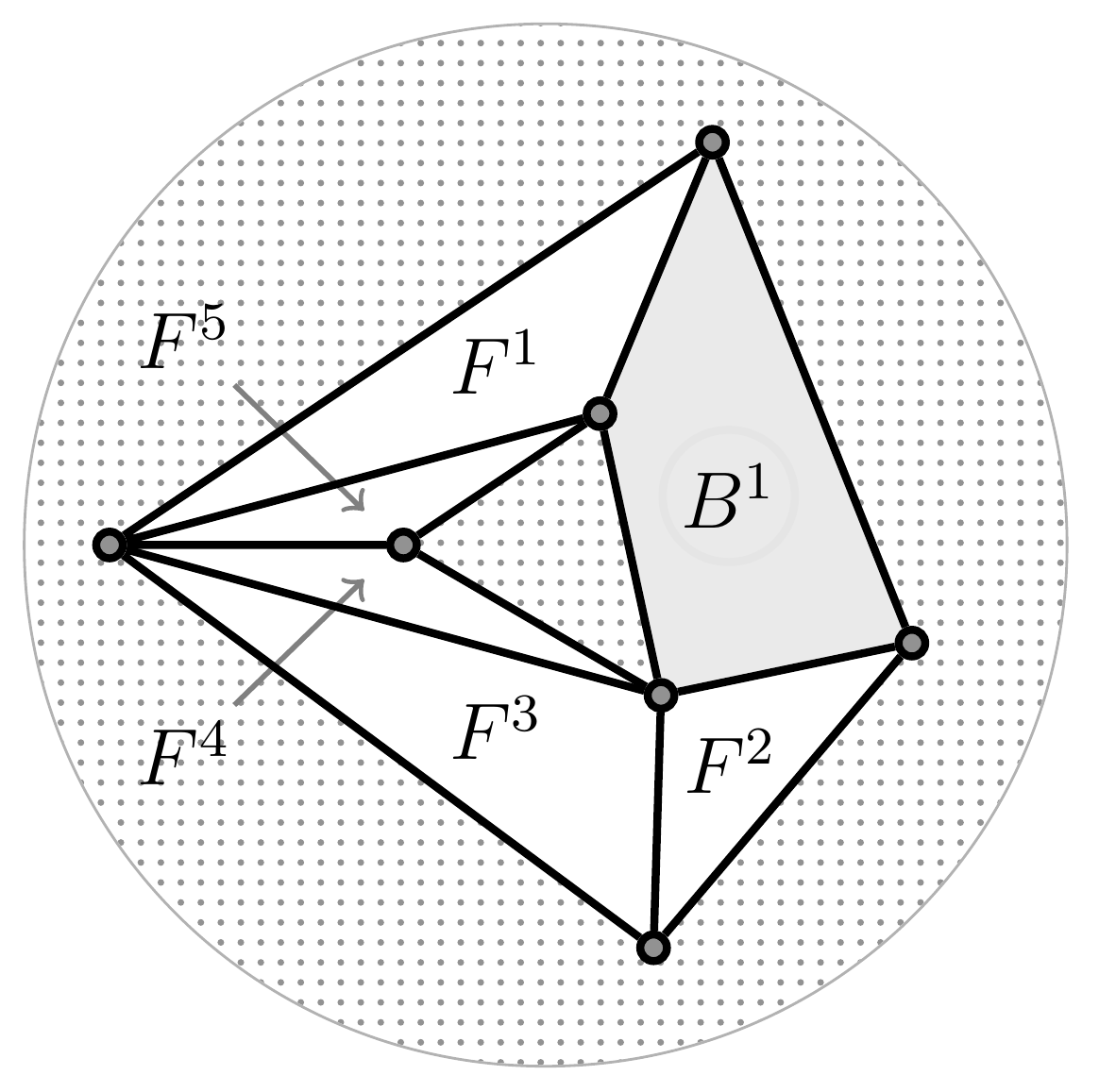}} \hspace{0.2in}
\subfloat[$G^M(\poly)$]{\label{fig:gussetedGM}\includegraphics[width=2in]{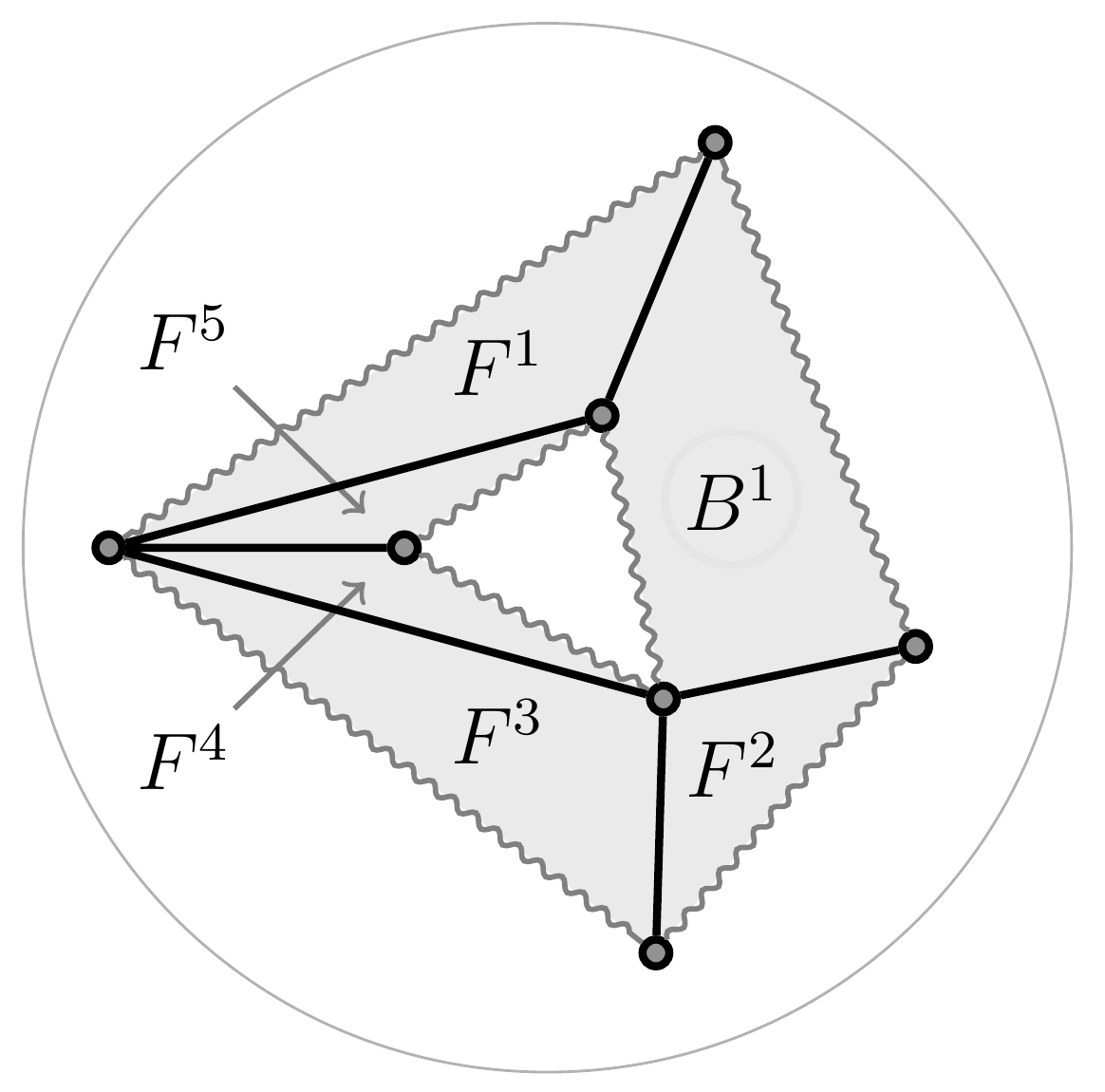}}
\caption{If we have a vertex in contact with two holes as in (a), the graph $G^{M}(\poly)$ shown in (b) does not track the motion of $\P$. The graph $\poly'$ shown in (c) has an isomorphic space of stresses, and now $G^{M}(\poly')$ shown in (d) does capture the motions. \label{fig:gussetNeed}}
\end{center} 
\end{figure}

\begin{definition}
{\rm A block and
hole polyhedra $\poly$ is }separated {\rm if there is at most  one hole at any given vertex.  
That is, there are no} contact
vertices {\rm where several holes meet.  }
\end{definition}
 
 \begin{proposition}  If a polyhedron $\poly$ is separated, then the spaces ${\mathcal M}(G_{S}(\poly,p))$ and
 ${\mathcal M}(G^{M}(\poly,\p))$ are isomorphic. 
 \label{prop:separatedStressMotion}
 \end{proposition}
 
 \proof We will only sketch the outline of this connection, as we have not developed the full background of
 infinitesimal motions for bar and joint frameworks, or even the velocities at points induced by 
 the centers of motion.  The result is implicit in \cite{crapowhiteley}.
 
 The general observation is that every motion of the bar and joint framework will induce velocities which 
 correspond to the centers of motion for the rigid faces.  There is an injection from 
 ${\mathcal M}(G_{S}(\poly, \p))$ into ${\mathcal M}(G^{M}(\poly, \p))$.
 
 If the polyhedron is separated, then each motion in ${\mathcal M}(G^{M}(\poly, \p))$ induces 
 a set of unique velocities for each of the vertices in $\p$, and therefore a unique corresponding motion
 in ${\mathcal M}(G_{S}(\poly, \p))$.  This completes the isomorphism. 
 \qed
 
Figure \ref{fig:example} shows the two graphs $G^{M}$ and $G_{S}$ for both a block and 
hole polyhedron $\poly$ and its swapped form, $\spoly$. It is the connections among these four structures which will be at the heart of the analysis in Section \ref{sec:motionsAndStressesSeparated}. 

We note that the swapping induces a map of the related graphs: $G^{M}(\ov{\ov{\P }}) = G^{M}(\P )$.  
Moreover $G_{S}(\ov{\ov{\P }})$ is equivalent to $G_{S}(\P)$.  
That is, they are the same graphs up to an interchange of which isostatic 
subframeworks are inserted into the blocks of $\P$. 

Our analysis will build up the complete set of  connections for separated block and hole polyhedra (Section \ref{sec:motionsAndStressesSeparated}).  
In Section \ref{sec:generalPolyhedra} we will describe an operation which takes a general (non-separated) block and 
hole polyhedron to an extended (gusseted) block and hole polyhedron
with equivalent spaces of stresses and motions as bar and joint frameworks. 
\\

\begin{example}  A {\it tower} is a block and hole polyhedra with one $s$-gon block and one $t$-gon hole. 
In this example, we consider several of the towers described in \cite{wfsw}, and investigate the impact of swapping on
the spaces of self-stresses and motion to get an idea of what to anticipate in the later proofs.  
We will assume there are $\min(s,t)$ edge disjoint paths between this block 
and hole, in order to give simple predictions for the anticipated 
spaces of stresses and first-order motions (see \cite{wfsw}). (See
Figure \ref{fig:example})  
\begin{enumerate}

\item  If $s > t$, then we have implicitly started with a triangulated sphere (with $|E|=3|V|-6$), 
removed $t-3$ edges from the 
hole and added $s-3$ edges to make the block polygon (already triangulated) into an isostatic subgraph.  The net effect 
is that $|E'| = 3|V|-6 + [s-t]>3|V|-6$.  This will have more bars than needed for first-order rigidity, so we will have 
a space of self-stresses of dimension at least $s-t$.  In generic realizations, 
we will have only the trivial motion assignment.   

\item  If $s < t$, then we have started with a triangulated sphere, removed $t-3$ edges from the 
hole and added $s-3$ edges to make the block polygon (triangulated) into an isostatic subgraph.  The net effect 
is that $|E'| = 3|V|-6 - [t-s]< 3|V|-6 $.  This will have fewer bars than needed to resolve all equilibrium loads, 
so we will have
a space of internal motions of dimension at least $t-s$.  In generic realizations, the bars will be independent, and we
will 
have only the zero dimensional space of trivial self-stresses.  

This structure is really a swapped polyhedron from case (i), where $s=t$ and $t=s$. 
Translating, we see that the dimension of the space of internal motions after swapping is the same dimension as the 
space of self-stresses before swapping, and that the space of self-stresses after swapping is the same dimension 
as the space of internal motions before swapping.  

\item  If $s=t$, then making the hole removes $s-3$ edges from a triangulated sphere, and adds $s-3$
edges to make the block polygon (triangulated) into an isostatic subgraph.  We return to a
collection of edges with the desired $|E|=3|V|-6$.  In generic position, this will be isostatic with only the trivial
self-stress, and only the trivial motion assignment.   The swapped polyhedron will have the same overall counts, and the
dimensions of stresses and motion assignments will match after swapping. 

\end{enumerate}
\end{example}

From the example, we can anticipate that swapping will take internal motions of the original block and hole 
polyhedra to self-stresses of the 
swapped structure, both generically, and in special position geometrically.   
This is a good set of elementary examples to keep in mind in the next few sections.

%%%%%%%%%%%%%%%%%%%%
%%%%%%%%%%%%%%%%%%%%
%%%%%%%%%%%%%%%%%%%%%%%%%%%%%%%%%%%%%%%%%%%%%%%%%%%%%
\section{Motions and Stresses in Separated Spherical Structures} \label{sec:motionsAndStressesSeparated}
%%%%%% Elissa  first draft August 10, 
% Revisions Wendy aug 18
% small change Walter Aug 31
% Revision Wendy Sept 6
% Small change Wendy Jan 25
% revision Walter February 18
%revision Walter August 2009, October 02

%%%%%%%%%%%%%%%%%%%%%%%%%%%%%%%%%%%%%%%%%%%%%%%%%%%%%%
In this  section, we prove our main result for the case of {\it separated block and
hole polyhedra} in which there is at most one block and at most one hole at any given vertex.  This is sufficient to guarantee that the graphs 
$G^{{M}}(\poly)$, as well as $G^{{M}}(\ov\poly)$, are connected and that these graphs provide full information about the rigidity of
the underlying block and hole polyhedron.  While there are block and hole polyhedra with contact vertices, it is helpful to cluster
all of the extensions to more general block and whole polyhedra into the next section, where a single added technique captures all these extensions. 

As our results are geometric in nature, we will be working with specific realizations of block and hole polyhedra as bar and
joint frameworks (for $G_{S}(\P)$) or as hinged panel structures (for  $\GM(\P)$).  When a projective embedding $\p$
of the graph is given, we write $\GM(\P, \p)$ or  $G_{S}(\P, \p)$ to denote the particular configuration in $\PS^3$
viewed as a body and hinge structure or a bar and joint framework, respectively. Note that the key rigidity properties,
namely static rigidity, infinitesimal rigidity, and independence, are projectively invariant; (see\cite{crapowhiteley}, \S4.2).

In a specific geometric realization, we will restrict the choices of the block subgraph in $G_{S}(\P, \p)$ to those which are isostatic 
in that realization. For this reason, in order to have some isostatic subgraphs available, 
we will assume that the configurations used do not have all vertices of a block face (and because of
swapping, of a hole) collinear.  We call such configurations {\it block and hole general position}.   

We develop the proof of the main theorem through two propositions. 

\begin{proposition} Given a separated block and hole polyhedron $\mathcal{P}$, there is 
injective map from the
stresses of the bar and joint framework $G_S(\ov\poly, \p)$ to the motion assignments of the hinged panel structure of the
swapped polyhedron at the same configuration, $\GM({\mathcal{P}}, \p)$.
\label{prop:stressesToMotionsSeparated}
\end{proposition}

\begin{proposition} Given a separated block and hole polyhedron ${\mathcal{P}}$, there is an injective map
from the motion assignments of the hinged panel structure $\GM({\mathcal{P}}, \p)$ to the stresses of the swapped
block and hole structure as a bar and joint framework at the same configuration,  $G_S(\ov{\mathcal{P}}, \p)$.
\label{prop:motionsToStressesSeparated}
\end{proposition}

Together, these propositions and their proofs will yield the main theorem for separated polyhedra:

\begin{theorem} 
Given a separated block and hole polyhedron, $\mathcal{P}$, there is 
an isomorphism between
 the space of 
motion assignments of the swapped block and hole structure $\GM(\ov{\mathcal{P}},\p)$ as 
a hinged panel polyhedron and
 the space of stresses of the bar and joint framework at the same configuration, $G_S(\mathcal{P}, \p)$. That is, $\mathcal S(G_S(\poly, \p)) \simeq \mathcal M(G^M(\spoly, \p))$.
 \label{thm:mainSeparated} 
\end{theorem} 

%%%%%%%%%%%%%%%%%%%%%%%%%%%%%%%%%%%%%%%%%%%%%%%%%%%%%%
\subsection{From Stresses to Motions in Separated Block and Hole Polyhedra} \label{sec:stressesToMotions}
%%%%%%%%%%%%%%%%%%%%%%%%%%%%%%%%%%%%%%%%%%%%%%%%%%%
%%%%%% ELISSA  first draft August 10,  
% more revisions ER, Aug 17
% Revisions Wendy aug 18
%  Revisions Walter February 18 
%  Revisions Walter August 8, 2009, October 02
%%%%%%%%%%%%%%%%%%%%%%%%%%%%%%%%%%%%%%%%%%%%%%%%%%%%%

\noindent {\bf Proof of Proposition \ref{prop:stressesToMotionsSeparated} for separated block and hole polyhedra.} 

Let $\spoly$ be a separated block and hole polyhedron with blocks $\blocks = \{\spolyblock^1, \dots 
\spolyblock^\ell\}$ and
holes $\holes = \{\spolyhole^1, \dots \spolyhole^n\}$, and consider the graph $G_S(\spoly, \p)$, where $\p$ is some
embedding of the graph into $\PS^3$.

Suppose we have a non-trivial stress $\Lambda$ on $G_S(\spoly, \p)$, given by scalars $(\dots,\lambda_{ij},\dots)$. 
 From the stress, we define a motion assignment on the swapped framework $\GM({\poly}, \p)$ by setting $\omega^{km} =
\lambda_{ij}$ if $\ijkm$ is an edge patch.  This map ignores the scalars within, or on the boundary of any block of 
$G_S(\spoly, \p)$ 
as a block $B_{\spoly}^{k}$ in $G_S(\spoly, \p)$ becomes the hole $H_{\poly}^k$ in the
swapped polyhedron and the edges are not in $\GM(\poly,\p)$

To show that the new set of scalars $(\ldots, \omega^{mk}, \ldots)$ is indeed a valid motion assignment to the 
hinges $(\dots, \Db^{mk}, \dots)$, 
we first prove that there is exactly one scalar for each edge; that is,  $\omega^{mk} = \omega^{km}$. 
To see this, note that the scalar assignment $\lambda_{ij}$ to the bar $\p_i\p_j$
must be balanced by the scalar $\lambda_{ji}$ on the bar $\p_j\p_i$. That is, 
$ \lambda_{ij} (\p_i \p_j) = -\lambda_{ji} (\p_j \p_i). $
Since $\p_i \p_j = - (\p_j \p_i)$, $\lambda_{ij} (\p_i \p_j) = \lambda_{ji} (\p_i \p_j)$, and 
it follows that  $\omega^{mk} = \lambda_{ij} = \lambda_{ji} = \omega^{km}$.

Next, we show that given any simple closed cycle 
$\mathcal{Z}=F^0 \Db^{01} F^1 \Db^{12} F^2 \cdots F^k \Db^{k0}$ of panels and hinges
in the swapped polyhedral structure $\GM(\poly, \p)$, we have $$\sum_{\Db^{km}\in \mathcal{Z}}
\omega^{km} \Db^{km} = \0$$ where the sum is over each hinge in the cycle.  Note that such a cycle necessarily 
does not cross any holes as faces.  

Because $\GM(\poly)$ is a part of the spherical polyhedron, any simple closed 
cycle of panels and hinges (possibly including blocks)
will disconnect the polyhedron, and the graph $G_{S}(\spoly)$.  
Since $\omega^{km} = \lambda_{ij}$, $\Db^{km} = \p_i \p_j$, and the collection of 
edges $\{v_{i},v_{j}\}$ is a cut set, 
Theorem \ref{thm:cutset} yields: 

\begin{equation}
\begin{split}
\sum_{\Db^{km} \in \mathcal{Z}} \omega^{km} \Db^{km}  &=  
\sum_{i = 1}^n    \left( \sum_{\substack{\{j| \Db^{km} \in \mathcal{Z}, \\ \ \ \ \ \ \ \Db^{km} = \q_i\p_j\}}} \lambda_{ij} \p_i \p_j \right)\ \ \ \ \ \hfill \\
 &= 0. \ \ \ \ \ \ \ \ \ \ \  \ \hfill 
\end{split}
\nonumber
\end{equation}

Since both conditions of a motion assignment are satisfied, a stress in a bar and joint framework 
induces a motion assignment in the swapped panel polyhedron.  

To see that this injective, we need to check that distinct self-stresses induce distinct motion
assignments.  Suppose there are two distinct self-streses $\Lambda$ and $\Lambda'$ with  $\lambda_{ij}\neq \lambda'_{ij}$ for some block edge, and all non-block edges have trivial (zero) scalars. Since the polyhedron is separated, no two blocks share
a vertex and hence the self-stress must be contained on a single block.  But each block has an
isostatic framework (independent) and therefore contains no self-stress, which provides the contradiction.  We conclude that the self-stresses $\Lambda$ and $\Lambda'$ must differ on 
some edges which induce distinct scalars for the motion assignment. 
\qed

The correspondence is actually a linear transformation of the two vector spaces,
 as the sum of two stresses goes to the sum of the 
 corresponding motion assignments and a scalar multiple of a stress goes to the scalar multiple 
 of the motion assignment. 
 
 \begin{corollary}
  Given a separated block and hole polyhedron $\mathcal{P}$, there is 
injective linear transformation from the space of
stresses of the bar and joint framework $G_S(\mathcal{P}, \p)$ to 
the space of motion assignments of the hinged panel structure of the
swapped polyhedron at the same configuration, $\GM(\ov{\mathcal{P}}, \p)$.
\label{cor:stressesToMotionsSeparatedVS}
 \end{corollary}

%%%%%%%%%%%%%%%%%%%%%%%%%%%%%%%%%%%%%%%%%%%%%%%%%%%
%%%%%%%%%%%
\subsection{From Motions to Stresses in Separated Block and Hole Polyhedra} \label{sec:motionsToStresses}
%%%%%% ELISSA  first draft August 10, 
% more revisions ER, Aug 17
% Revisions Wendy aug 26
% Revisions Walter February 18
%  Revisions Walter August 8, 2009, October 02
%%%%%%%%%%%%%%%%%%%%%%%%%%%%%%%%%%%%%%%%%%%%%%%%%

We now give the proof of the converse result.\\ 

\noindent{\bf Proof of Proposition \ref{prop:motionsToStressesSeparated}  for separated block and hole polyhedra} 

Let $\poly$ be a separated block and hole polyhedron with blocks $\blocks = \{\polyblock^1, \dots \polyblock^\ell\}$ and
holes $\holes = \{\polyhole^1, \dots \polyhole^n\}$, and consider the graph $\GM(\poly, \p)$, where $\p$ is some
embedding of the graph into $\PS^3$. 
%%Let the hinge, $\Db^{km}$, separate face $k$ from face $m$.  
%% Recall that if $(k, m; i, j)$ is an edge patch, then $\Db^{km} = a_i a_j$. 

Suppose we have some non-trivial motion assignment $(\dots \omega^{km} \dots )$ on the graph $\GM(\poly, \p)$. 
Consider the swapped block and hole polyhedral framework $\spoly$, and set 
$\lambda_{ij} = \omega^{km}$ whenever $\ijkm$ is an edge patch. 
The result is an assignment of scalars to some of the bars of $G_S(\spoly, \p)$. 
We show that this can be extended to a stress on all the bars of $G_S(\spoly, \p)$. 
In particular, we show that we can assign scalars to the remaining bars such that
for each vertex of $\spoly$, Equation (\ref{eqn:selfStressCondition}) is satisfied.  

Since $\poly$ is separated, $\spoly$ is also separated.  Therefore any vertex, $v_i$ of $\spoly$ is 
one of the following two distinct types: 
it is either adjacent to a single block (called a {\it block vertex}), 
or $a_i$ is adjacent to no block of $\spoly$.  

In the latter case, there is a simple closed cycle of panels and oriented hinges in $\GM(\poly)$, 
$\mathcal{Z}=F^0 \Db^{01} F^1 \Db^{12} F^2 \cdots F^k \Db^{k0}$ around $v_i$.  
Since $\lambda_{km} = \omega^{ij}$, $\Db^{km} = \p_i \p_j$, and $\sum \omega^{km}\Db^{km} = \0$  
for the cycle around $v_i$,  
$$\sum_{i = 1}^n \left( \sum_{\substack{\{j| \Db^{km} \in \mathcal{Z}, \\\ \ \ \ \ \ \Db^{km} = \q_i\p_j\}}} \lambda_{ij} \p_i\p_j \right)= \0$$
 holds in $\spoly$.  Hence, Equation (\ref{eqn:selfStressCondition}) holds for 
all non-block vertices.  

%To see that Equation \ref{eqn:selfStressCondition} holds for a hole vertex $v_h$ of $\spoly$, note that any hole 
%$\spolyhole^{k}i$ was formerly the
%block $\polyblock^{k}i$, with the motion scalars only assigned to edges in $\GM(\poly)$ - that is 
%to the boundary edges assigned of the block, 
% $\polyblock_i$ necessarily zero 
%(see Figure~7). 
%The simple closed cycle of panels and hinges around the vertex $a_h$ in $\GM(\poly, \p)$ 
%satisfies cycle condition % \ref{motion condition} 
% with the interior edges of the block making no contribution to this sum. 
%Therefore, in the swapped bar and joint framework $G_S(\spoly, \p)$, the hole vertex $v_h$ 
%will have the property that 
%$ \displaystyle \sum_{\{ j: (h,j) \in   E\}} \lambda_{hj}\p_h \p_j = \0$. 

%The block vertices in the original panel polyhedron $\GM(\poly, \p)$ satisfy the cycle of faces condition (reference something in intro to 
% motions section here), 
%with the interior edges of the block making no contribution to this sum. 
%Therefore the corresponding hole vertex $h$ in the swapped bar and joint framework $G_S(\spoly, \p)$ will have the property that the sum 
%$ \displaystyle \sum_{\{ j: (i,j) \in   E\}} \lambda_{ij}h_i a_j = 0$, where the $a_j$'s are the vertices that are connected to the hole 
%vertices by single edges. \\

It remains to extend the scalars $(\ldots, \lambda_{ij}, \ldots)$ to the edges of the blocks in $\spoly$, 
and to show that the block vertices satisfy Equation (\ref{eqn:selfStressCondition}).  
Because $\poly$ is a separated block and hole polyhedron, 
there is a simple closed cycle of panels and hinges,
 $\mathcal{Z}= (F^0 \Db^{01} F^1 \Db^{12} F^2 \cdots F^k \Db^{k0})$, 
around the boundary of every hole in $\GM(\poly, \p)$ (i.e. every block of $\spoly$).  
The cycle satisfies the cycle condition %Equation~\ref{motion condition},
$\sum_{\Db^{km}\in \mathcal{Z}} \omega^{km} \Db^{km} = \0$.  
Consider the hole $H_{\poly}^i$ with vertices $J = \{b_1, \ldots, b_n\}$ and let 
$\mathcal{Z}$ be the set of oriented hinges into the hole vertices $b_i$ from adjacent non-hole vertices $a_j$.
Under swapping, this property of the holes of $\poly$ translates to a property about the blocks in $\spoly$.  Indeed,  
$\spolyblock^i$ satisfies 
\begin{equation}\label{eqn:block2}  
\displaystyle \sum_{i = 1}^n \left( \sum_{\substack{\{j| \Db^{km} \in \mathcal{Z}, \\\ \ \ \ \ \ \Db^{km} = \q_i\p_j\}}} 
\lambda_{ij} \q_i \p_j \right) = \0
\end{equation}

Equation~(\ref{eqn:block2}) implies that the forces on the block vertices form an equilibrium load.
Since each block is built as an isostatic subframework, the block is statically rigid, and therefore its edges will resolve any
equilibrium load on the block.   Furthermore, because the block is isostatic and therefore independent, there is in fact a {\it unique}
assignment of scalars to the block edges that resolves this load. 
These scalars are the missing $\lambda_{ij}$'s.
We now have exactly one scalar for every edge of $G_S(\spoly, \p)$.  
The resolution of the stress in the block guarantees that the equilibrium condition 
(Equation (\ref{eqn:selfStressCondition})) is satisfied for each of the block vertices.

Note that the zero motion assignment goes to the zero self-stress, as expected. 
If we have two different motion assignments with $\omega^{k,m}\neq \theta^{k,m}$, 
clearly the induced self-stresses are zero, so the map is injective. 
\qed 

Again, this map is a linear transformation.  

\begin{corollary}  Given a separated block and hole polyhedron ${\mathcal{P}}$, there is an injective map
from the space of motion assignments of the hinged panel structure $\GM({\mathcal{P}}, \p)$ to the space of stresses of the swapped
block and hole structure as a bar and joint framework at the same configuration,  $G_S(\ov{\mathcal{P}}, \p)$.
\label{cor:motionsToStressesSeparated}
\end{corollary}

%%%%%%%%%%%%%%%%%%%%%%%%%%%%%%%%%%%%%%%%%
\subsection{Main result and corollaries} \label{sec:mainResult}
%%% Walter Friday August 17           %%%
%%% Revised Wendy Sept 13             %%%
%%% Revised Walter February 18, 2008  %%%
%%% Revised Walter August 8, 2009 , October 02   %%%
%%%%%%%%%%%%%%%%%%%%%%%%%%%%%%%%%%%%%%%%%%

Section~\ref{sec:stressesToMotions} completed the proof of Proposition~\ref{prop:stressesToMotionsSeparated}
 for separated block and hole polyhedra and Section~\ref{sec:motionsToStresses} completed the proof of Proposition \ref{prop:motionsToStressesSeparated}.
 The corollaries extended these results to demonstrate isomorphism of the corresponding vector spaces. 

There are some geometric corollaries, as well as some generic corollaries of these two propositions and 
Theorem~\ref{thm:mainSeparated}. 

\begin{corollary}
The vector space of stresses on a separated block and hole polyhedral framework $G_{S}(\P,\p)$ is isomorphic 
to the vector space of
motion assignments of the swapped block and hole polyhedron ${\GM}(\ov{\P},\p)$.
\label{cor:mainTheoremVS}
\end{corollary}

%%%%%%%%%%%%%%%%%%%%%%
%%%%%%%%%%%%%%%%
\section {General Block and Hole Polyhedra} \label{sec:generalPolyhedra}
%%%%Walter Feb 16 08%%%%
%%%%%%%%%%%%%%%%%%%%%%%%%%%%

Following the proofs for separated block and hole polyhedra, in this section we extend the results to more general block and hole polyhedra that have adjacent holes or adjacent blocks. 
We assume that the configurations used in these geometric results do not have three adjacent vertices
 in any block or hole collinear, 
that is, that the configurations are in block and hole general position.  
This is not needed for separated polyhedra, where no vertex lies on more than one block or hole. 

%%%%%%%%%%%%%%%%%%%%%%%%%%%%%%%%%%%%%%%%%%%%%
%%%%%%%%%%%%%%%%%%
\subsection{Gussets to create separated block and hole polyhedra} \label{sec:gussets}
%%% Walter Friday August 17 %%%
%%% Revised Wendy Sept 13   %%%
%%%% Revised Walter Feb 18 %%%%%%%%%%%%%%%%%%%%%%%%%%%%%%
%%%%%%%%%%%%%%%%%%%%%%%%%%%%%%%%%
In this subsection, we will introduce {\it gussets};  gussets add some edges and triangles to separate the original
block and hole polyhedron, 
while extending the stresses and motions of the original polyhedron to isomorphic spaces on the extended polyhedron.  
In addition, gussetting and swapping are commutative operations (that is, first gussetting then swapping will yield
identical results to swapping then gussetting).  
The gussetting procedure will allow us to extend Theorem~\ref{thm:mainSeparated} to general spherical block and hole polyhedra. Figure~\ref{fig:gussetNeed} provides an intuition for the gusseting procedure.

Throughout this section, we assume $\p$ is in block and hole general position for a given polyhedron $\poly$.
Thus, in the following constructions a set of three vertices $\{u_{-1}, u, u_{+1}\}$, 
where $\{u_{-1}, u\}$ and  $\{u, u_{+1}\}$ are adjacent vertices along the same block or hole,
 are not collinear. We refer to the embedded graph of the polyhedron $G(\poly, \p) = G(\p)$ 
as a block and hole polyhedron, or simply as a polyhedron.  
  
\begin{definition} {\rm Given a non-separated vertex $u$ on a hole $H$ of a block and hole polyhedron $G(\p)$ 
such that the two pairs $u_{-1}, u$, and $u, u_{+1}$ are adjacent vertices along $H$, 
a  {\it hole-gusseted polyhedron}, $G'(\p')$, is the extension of $G$ by the addition of a new 3-valent vertex $v_{u}$ which is attached by
three edges, $(v_{u},u_{-1})$, $(v_{u}, u)$, and $(v_{u},u_{+1})$,
creating the new hole $H'$.  Geometrically, we ensure that the new vertex is assigned a position $p_{v_{u}}$ in $\p'$
which is not coplanar with the three attaching
vertices. (See Figure~\ref{fig:holeGusset}.)}
\end{definition}

\begin{figure}
\begin{center}
\includegraphics[width=4in]{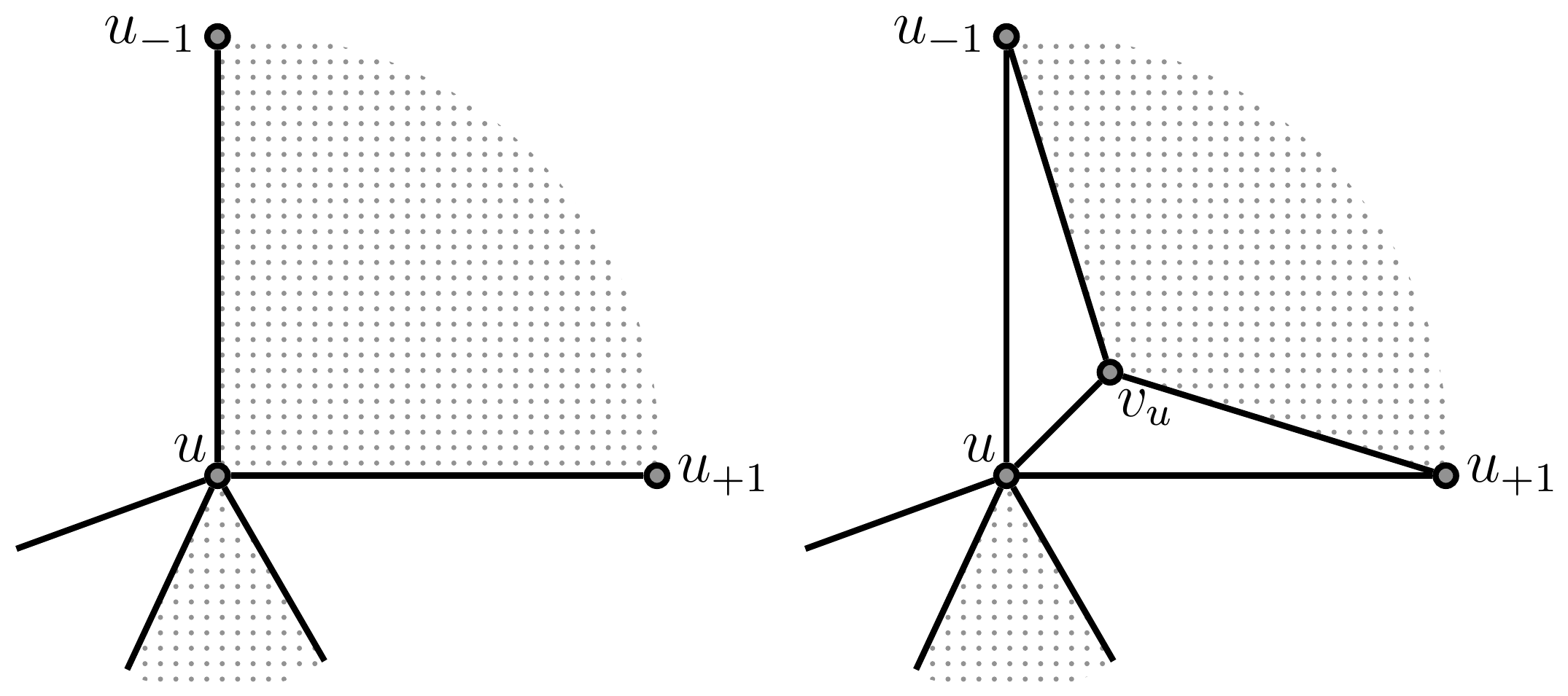} 
\caption{A non-separated hole vertex (left) becomes a non-contact vertex with the addition of a gusset (right). \label{fig:holeGusset}}
\end{center}
\end{figure}

This construction is essentially what is traditionally called 3-valent addition for bar and joint frameworks.  
However, what is different
here is that we also extend the underlying polyhedron and regard $\{u_{-1}, u, v_{u}\}$ and  
$\{v_{u},u, u_{+1}\}$ as two new triangular surface faces of the polyhedral structure.  
Notice that the addition of a gusset does not change the size of the hole, it just shifts the hole away from $u$ to $v_{u}$,
creating a new hole $H'$.  

\begin{proposition} 
Provided that $\p_{v_{u}}$ is not coplanar with 
$\p_{u_{-1}}$, $\p_u,$ and $\p_{u_{+1}}$  as we create $\p'$, 
there is a linear isomorpism between the space of stresses of $G(\p)$ and
the space of stresses of $G'(\p')$ with the added hole-gusset, and between the motion assignments
 of $G(\p)$ and
the motion assignments of $G'(\p')$ with the added hole-gusset.  
\label{prop:gussetIsomorphismHole}
\end{proposition}

\noindent{\bf Proof}.  The new $3$-valent vertex (not coplanar with the vertices of attachment) cannot support a non-zero stress on its edges, so the
space of stresses is unchanged \cite{taywhiteley}.  This shows the isomorphism of the stresses.

We have added a 3-valent vertex, which is independent.  Therefore, by basic operations on the underlying framework,
the space of first-order motions is also unchanged \cite{taywhiteley}.  Given any first-order motion of $G(\p)$ there is a unique velocity
assigned to $p_{v_{u,H}}$ which extends the first-order motion.   This in turn gives an extension of the motion
assignment, leaving all previous scalars unchanged and generating new motion assignment scalars for the new edges
between faces of the polyhedron: $\{u_{-1}, u\}$, $\{u, v_{u}\}$ and $ \{u, u_{+1}\}$.
Again the uniqueness gives us the one-to-one correspondence of the spaces of motion assignments.
\qed

Next, we assume that $u$ is a non-separated vertex on a block, $B$.  For defining a block gusset we have three goals:
\begin{enumerate} 
\item we should make a new block which is 
no longer directly in contact with the vertex $u$, by inserting some triangular faces and a new vertex $v_{u}$; 
\item this addition of a new vertex should be
equivalent, after swapping, to adding a gusset to the hole in the swapped structure, and
\item  the isostatic
graph used for the block should be unchanged, with the new vertex $v_{u}$ replacing
the vertex $u$.  
\end {enumerate} 
\begin{figure}
\begin{center}
\includegraphics[width=4in]{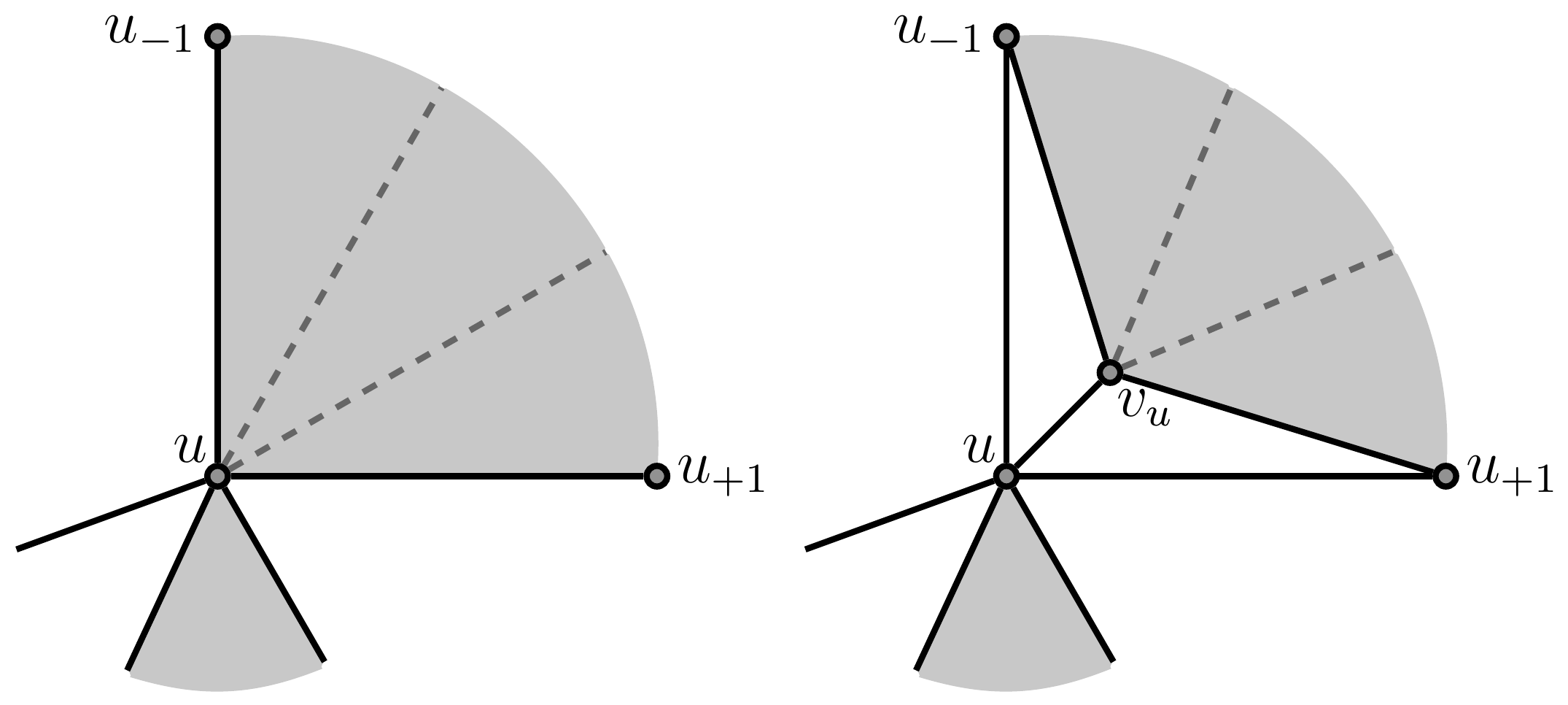} 
\caption{A gusset added at a non-separated block vertex. \label{fig:blockGusset}}
\end{center}
\end{figure}

The second objective gives us a clear picture of what we should do.  

\begin{definition}{\rm  Given a non-separated vertex $u$ on a block $B$ of a block and hole polyhedron $G(\p)$, 
the  {\it block-gusseted
polyhedron} $G'(\p')$ is the extension of $G$ by the vertex split of $u$ to $\{u,v_{u}\}$ 
along the edges $\{u_{-1}, u\}$ and $\{u, u_{+1}\}$, 
with all the block edges adjacent to $u$ in $G(\p)$ becoming adjacent to $v_{u}$ in $G'(\p')$, instead. (See Figure~\ref{fig:blockGusset}.) \\

Geometrically, we ensure that the new vertex is assigned a position $p_{v_{u}}$ in $\p'$
which keeps the isostatic block framework of $B'$ isostatic, and is also not coplanar with 
the three vertices of the split.   This is always possible, 
by general considerations of avoiding the thin set of special positions for any particular graph 
vertices; cf \cite{wwI}. We call such a position {\it general for $(u, B)$}. }
\end{definition}

\begin{proposition} 
There is an isomorphism between the space of stresses of $G(\p)$ and
the space of stresses of $G'(\p')$ with the the added block-gusset, and between the space of 
motion assignments of $G(\p)$ and
the space motion assignments of $G'(\p')$ with the the added block-gusset. 
\label{prop:gussetIsomorphismBlock}
\end{proposition}

\proof
We have effectively replaced a general isostatic block framework on $B$
with a specific type of isostatic block framework in which we have added one new vertex $v_{u,B}$
and have ensured that $u$ is $3$-valent in the subframework, 
attached only to $u_{-1}$, $u_{+1}$ and $v_{u}$, and not coplanar with these points. 
% Effectively, we are replacing an isostatic block on $B$ with another isostatic block on $B$ 
% which happens
% to contain an additional vertex.  
By the Isostatic Substitution Principle, there
is no change in the space of stresses or motions of the structure.  
\qed
 
Now, given any vertex in contact with more than one hole or block, we insert gussets at this vertex into all but
one of
the holes and all but one of the blocks at that vertex. Repeat this for each such vertex, in sequence, and the resulting structure will
be a separated block and hole polyhedron, as shown in Figure~\ref{fig:gussetCommutative}.  

\begin{figure}
\begin{center}
\includegraphics[width=3.5in]{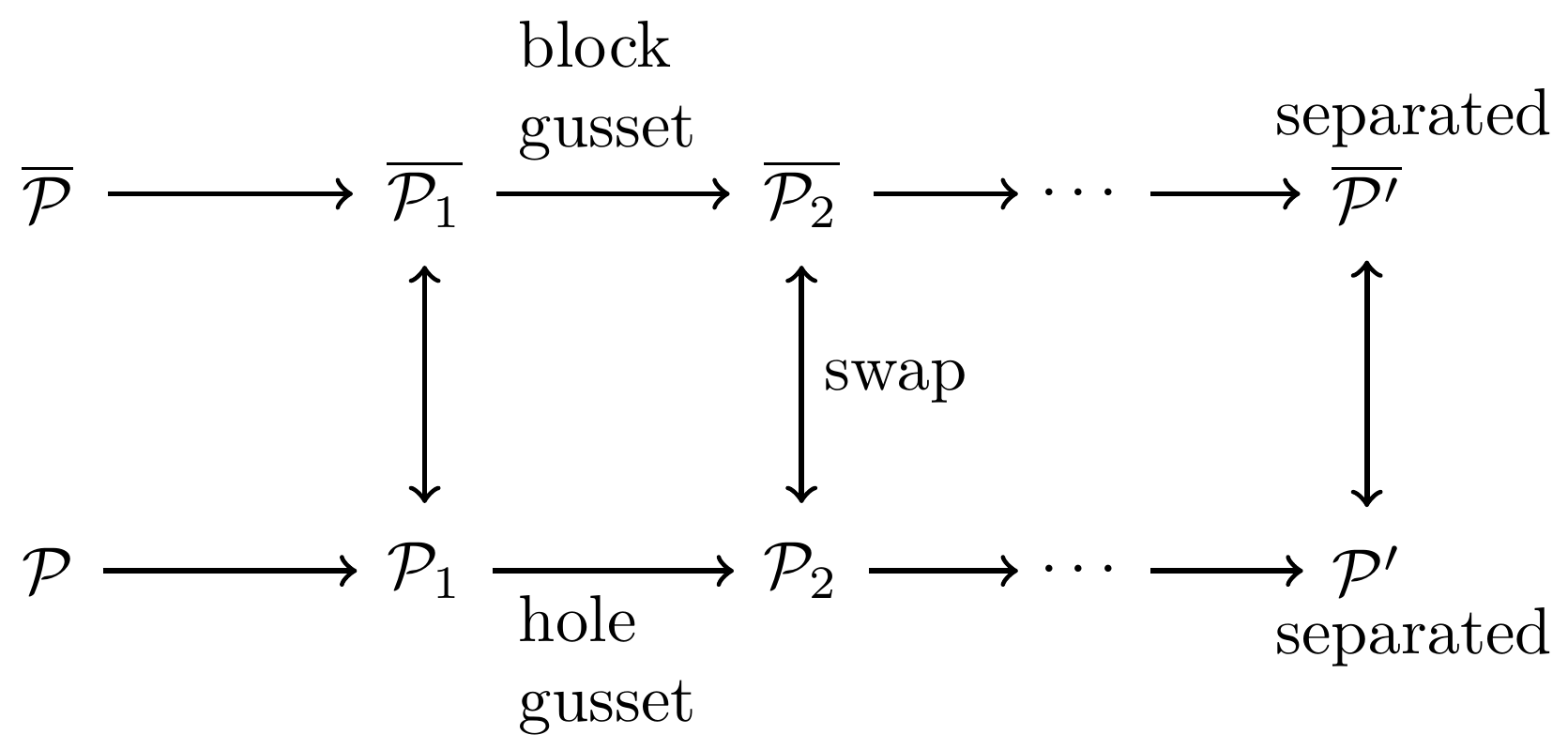} 
\caption{The gusseting and swapping processes.\label{fig:gussetCommutative}}
\end{center}
\end{figure}

\begin{lemma} Given a block and hole polyhedron $G(\P,\p)$, adding a hole-gusset at the vertex $u$ on a hole $H$ and swapping,
creates the same structure as swapping to $G({\ov \P},\p)$, and adding a block-gusset
at the vertex $u$ on the block $\ov{H}$, provided we assign the
same position to the added vertex, in creating $\p'$. 

Given a correspondence of the motion assignment on  $G'(\P',\p')$ with the stresses of $ G'(\ov{\P'},\p')$, 
there is an induced correspondence of the motion assignments on $G(\P,\p)$ 
with the stresses of $G(\ov{\P},\p)$.  
\label{lem:gussetCommutative}
\end{lemma}

Notice that we will need to pick the new point to be general for $(u, \ov{H})$, since the block puts more
restrictions on the location of the new vertex.

\begin{proposition}\label{prop:inducedGussetedPolyhedron}
Given a block and hole polyhedron $G(\P,\p)$, with $\p$ in general position for the blocks and holes, then there
is an induced 
gusseted polyhedron $G'(\P',\p')$ which is separated and has an isomorphic space of stresses and an
isomorphic space of motion assignments. 
\end{proposition}

\noindent{\bf Proof.} List the non-separated vertices of $G$, and for each such vertex, list the blocks and holes around
the vertex.   We proceed in turn, vertex by vertex, following this arbitrary sequence.  As we process a given vertex, we
add a gusset to all blocks and holes except the last one in the sequence at the vertex.   Given the sequence, the
gusseting is unique.  As we add new geometric points, we ensure that they are not coplanar with their attachments,
and moreover, we ensure that when added they are not collinear with any two of the other vertices of the block or
hole. As well, we ensure that they are in general position for the block in one of the swap-equivalent pairs. 

At each stage, the stresses and the motions are isomorphic to the previous polyhedron.
By induction and Propositions~\ref{prop:gussetIsomorphismBlock}, \ref{prop:gussetIsomorphismHole}, we create a final polyhedron and realization with a space of stresses and a space of motion
assignments which are isomorphic to the spaces of the original.  Since we gusseted all but one of the blocks or holes
at a non-separated vertex, this vertex is now separated in the extension, and all vertices added are also separated. 

Notice that the two gusseting processes are equivalent, under swapping (Lemma~\ref{lem:gussetCommutative}).  We conclude that, provided we use the same
sequence of non-separated vertices, and of block and hole faces at these non-separated vertices, we take a pair of swapped
block and hole polyhedra to a new pair of separated swapped block and hole polyhedra.
\qed

Gusseting is a general `trick' which can be applied whenever we wish to separate identified polygonal 
`holes' or isostatic `blocks' in polygons.  The results for a framework $G(\p)$, and its gusseted extension
$G'(\p')$ are local and do not depend on 
information about the larger scale topological patterns of the framework.  

%%%%%%%%%%%%%%
%%%%%%%%%%%%%%%%%%%%%%%%%%%%%%%%%%%%%%%%%
\subsection{Main result revisited} \label{sec:mainResultRevisited}
%%% Walter Friday August 17 %%%
%%% Revised Wendy Sept 13   %%%
%%% Revised Walter February 19  %%%
%%%%%%%%%%%%%%%%%%%%%%%%%%%%%%%%%%%%%%%%%%

Using Proposition~\ref{prop:inducedGussetedPolyhedron} and the techniques of Sections~\ref{sec:stressesToMotions} and \ref{sec:motionsToStresses} we can prove general case versions of Propositions~\ref{prop:stressesToMotionsSeparated} and \ref{prop:motionsToStressesSeparated}.  Together these results form the proof of the following theorem, which is a general version of Theorem~\ref{thm:mainSeparated}:

\begin{theorem} 
Given a block and hole polyhedron, $(\mathcal{P},\p)$ with $\p$ in general position, there is 
an isomorphism between
 the space of 
motion assignments of the swapped block and hole structure $\GM(\ov{\mathcal{P}},\p)$ as 
a hinged panel polyhedron and
 the space of stresses of the bar and joint framework at the same configuration, $G_S(\mathcal{P}, \p)$.
 \label{thm:mainGeneral} 
\end{theorem} 

\noindent{\bf Proof.}\\

\begin{figure}[h!] 
\begin{center}
\includegraphics[width=4.5in]{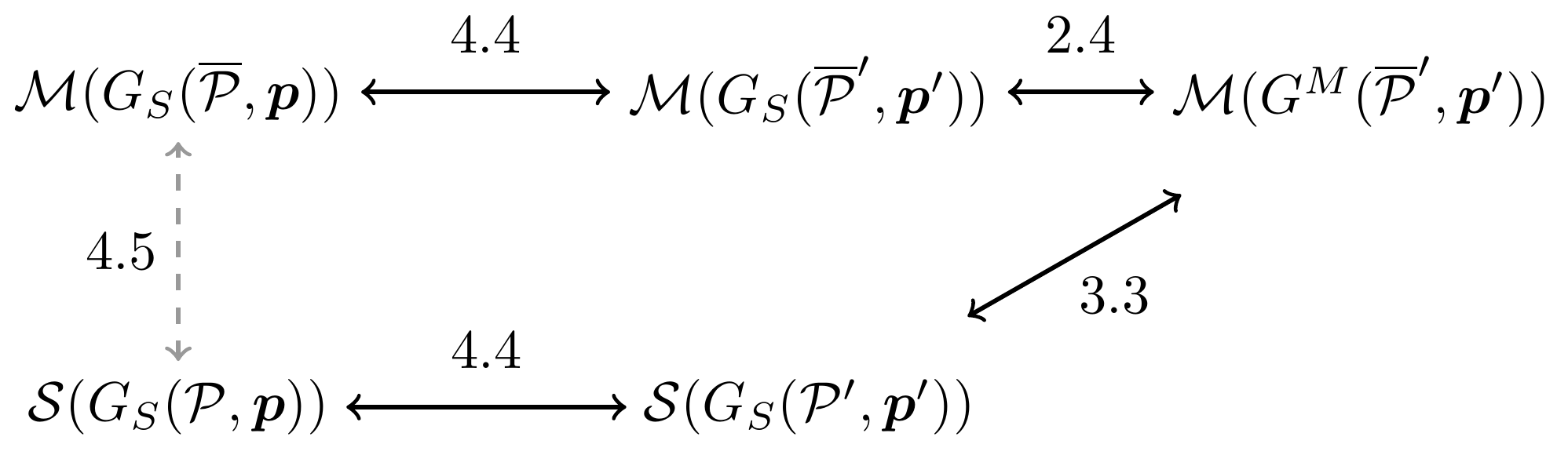} 
%\caption{Proof of the main theorem.\label{commutative}}
\end{center}
\end{figure}
Combining Proposition~\ref{prop:inducedGussetedPolyhedron} with Theorem~\ref{thm:mainSeparated} and Proposition~\ref{prop:separatedStressMotion} yields the result. 
\qed

There are some geometric and generic corollaries of these general results. %
%
%\begin{corollary}
%The vector space of stresses on a block and hole polyhedral framework $G_{S}(\P,\p)$ is isomorphic 
%to the vector space of
%motion assignments of the swapped block and hole polyhedron ${\GM}(\ov{\P},\p)$.
%\label{cor:mainGeneral}
%\end{corollary}
%
The results can be posed in the contrapositive, since first-order rigidity is equivalent to having only the zero 
motion assignment, and independence is equivalent to having only the zero stress.  This form is 
appropriate to a number of applications. 

\begin{corollary} Given a block and hole polyhedral framework $G(\P,\p)$ and the swapped polyhedral framework
 ${G}(\ov{\P}, \p)$, 
 \begin{enumerate}
\item[i)]  if  $G_{S}(\P,\p)$ is  geometrically isostatic then  ${G_{S}}(\P,\p)$ is geometrically isostatic at the same configuration,
\item[ii)] if  ${G_{S}}(\ov{\P}, \p)$ is geometrically independent 
then  $G^{M}(\P,\p)$ is geometrically first-order rigid at 
the same configuration, and
\item[iii)] if  ${G^{M}}(\P,\p)$ is geometrically first-order rigid then  
${G_{S}}(\ov{\P},\p)$ is geometrically independent.
\end{enumerate}
\label{cor:applicationsGeometric}
\end{corollary}

Since any generic configuration will also be in general position for all blocks and holes,  the geometric results
immediately transfer to generic results.

\begin{corollary} Given a block and hole polyhedron $\P$ and the swapped polyhedron $\ov{\P}$, 
\begin{enumerate}
\item[i)] $G^{M}(\poly)$ is generically rigid if and only if ${G_{S}}(\ov{\P})$ is generically independent, and
\item[ii)]  $G_{S}(\poly)$ is generically isostatic if and only if  ${G^{M}}(\ov{\P})$ is generically isostatic.
\end{enumerate}
\label{cor:applicationGeneric}
\end{corollary} 

There are two special cases which are contained in this general result, and are, implicitly or explicitly, in the prior literature. 
\\

%\noindent{\bf Example 4.2.1.}
\begin{example}
An extreme form of the block and hole polyhedron will have no `surface triangles' and no holes.
The polyhedron is composed only of blocks (some of which may 
happen to be triangles):  
 $\poly = (\B_{\poly})$, while the swapped polyhedron has no blocks:  $\ov{\poly} = (\Ho_{\ov{\poly}})$.  In this setting, the 
 key observation is that: 
$$\GM(\poly)=G_{S}(\ov{\poly})=G(\poly)=G(\ov\poly).$$   
The entire polyhedron is a body and hinge spherical polyhedron with hinges $G(\poly)$, and a spherical framwork with edges 
$G(\ov\poly)=G(\poly)$
This special form of the result is stated in \cite{crapowhiteley}, 
and is used explicitly in \cite{infp1} to prove variations on Cauchy's Theorem for convex polyhedra.  
The correspondence is also implicit (in Euclidean
terms) in some 
remarks  in \cite{alexandrov}.
\end{example}

%%%%%{Example 4.2.2}
\begin{example} As an intermediate case, we can have only blocks and one hole (no indentified surface triangles):
$\poly = (\B_{\poly}, \{H\})$.  As a hinge structure, this is a disc of rigid panels (blocks), leaving the `exterior' as 
a single hole.  (See Figure~\ref{fig:panelExample}.) The swapped structure $\ov\poly = ( \{B\}, \Ho_{\spoly})$ has one block, which we often think 
of as a rigid ground, and the rest is a bar and joint framework on the edges of the polyhedron. 
 The  two connections which are at the core of Section 3 still give an isomorphism between the 
motion assignments of $\GM(\poly)$ and the self-stresses of $G_{S}(\ov\poly)$.  
Such `panel discs' are encountered implicitly in a number of 
studies such as \cite{bipartite}, as well as some recent work on structures built on quadgraphs in discrete differential geometry.  

\begin{figure}
\begin{center}
 \subfloat[]{\label{fig:panelExampleA}\includegraphics[width=1.8in]{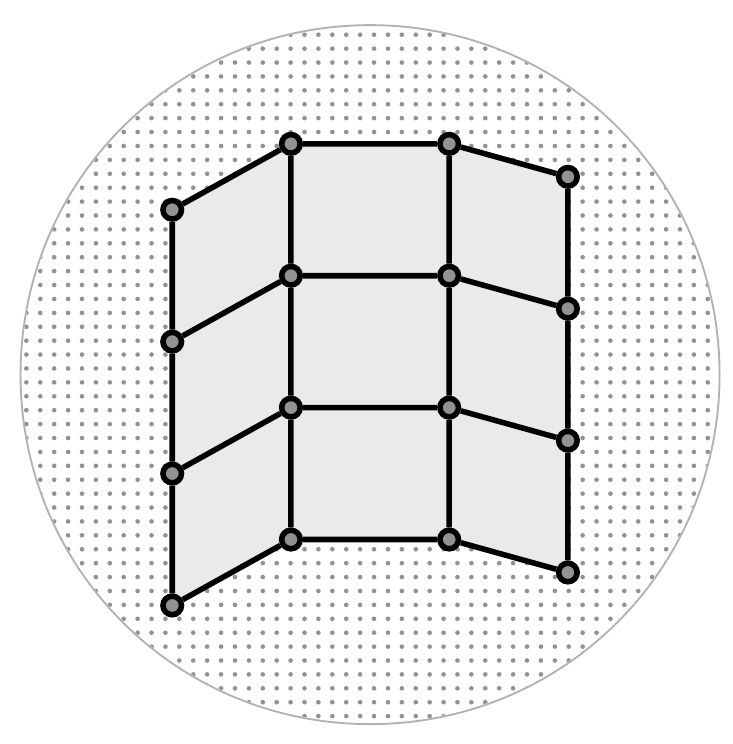}}  
  \subfloat[]{\label{fig:panelExampleB}\includegraphics[width=1.6in]{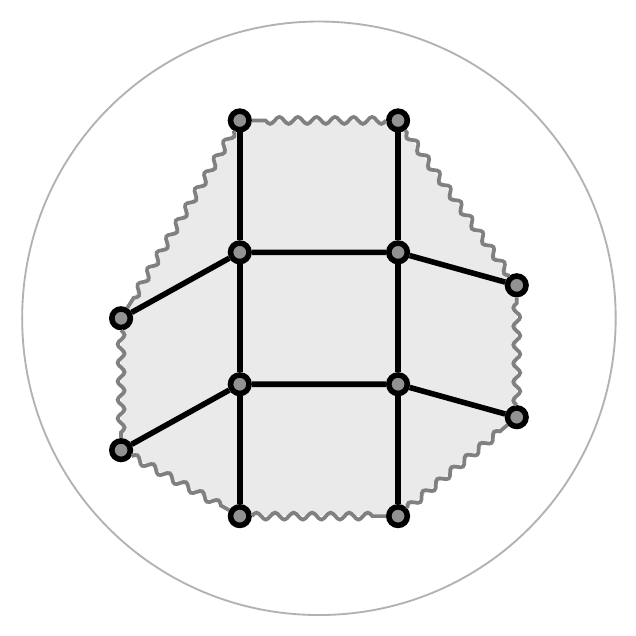}}   
 \subfloat[]{\label{fig:panelExampleC}\includegraphics[width=1.6in]{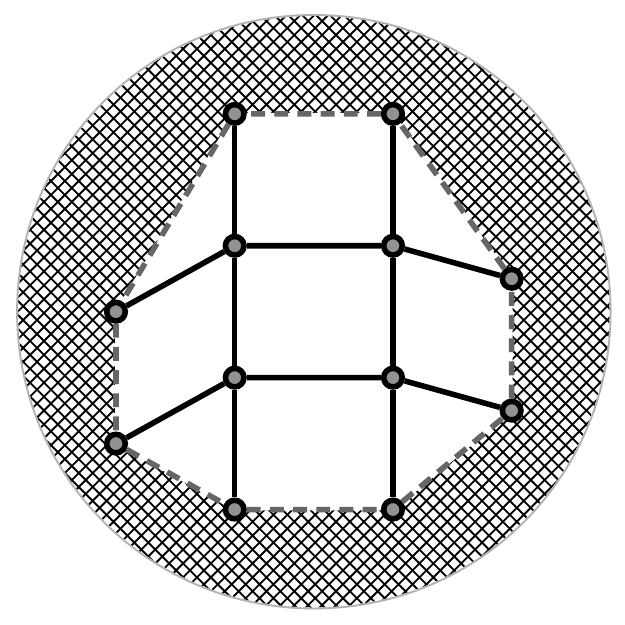}}   
 
\caption{(a) depicts a block and hole polyhedron $\poly$ consisting only of blocks. Figure (b) shows the graph $G^M(\P)$ (in which two valent vertices have been removed), and (c) depicts the graph $G^S(\spoly)$ of the swapped polyhedron. This graph can also be viewed as a pinned framework. \label{fig:panelExample} }
\end{center}
\end{figure}

\end{example}

%%%%%%%%%%%%%%%%
\section {Some Extensions and Further Work} \label{sec:extensions}
%%%%Walter Feb 16 08%%%%
%%%%%%%%%%%%%%%%%%%%%%%%%%%%

%%%%%%%%%%%%%%%%%%%%
\subsection{Projective interpretations} \label{sec:projective}
We have used projective coordinates for all of the geometric results in this paper. 
It is geometrically possible that that some of the `vertices' or even some of the `faces' will lie on the plane
at infinity in projective space.   These actually have valid static and kinematic interpretations, as noted
in \cite{crapowhiteley}.   For example, a joint at infinity, in our structure, will amount to a set of 
rigid pieces (triangles, blocks) which have hinges which are parallel, pointing to this infinite `vertex'.   
 A hinge line between two rigid pieces, at infinity, can be represented by a `slide hinge' between the two panels
 which leaves only the one degree of freedom between the rigid objects - that is, a translation perpendicular to all
 the planes through that line at infinity.   
 In this way, using the projective forms includes 
 some additional realizations which actually occur in mechanical engineering.

 A second general byproduct of the projective form and invariance of the statics and 
 centers of motion assignments 
 is that all of the geometric (and combinatorial) results work extend to the various metrics which can be extracted
 from the projective geometry:  Euclidean, spherical, and hyperbolic \cite{saliola-whiteley}.

%%%%%%%%%%%%%%%%%%%%%%%%%%%%%%
\subsection{Pure conditions for isostatic block and hole polyhedra} \label{sec:pureConditions}
From earlier work of White and Whiteley \cite{wwI}, we know that any generically isostatic
graph $G$ in $3$-space has an associated {\it pure condition}  $c_{G}$, which is a polynomial in the 
(projective) coordinates of the vertices.   This polynomial represents the special positions in the 
specific sense that geometrically $G(\p)$ is isostatic if and only $c_{G}(\p) \neq 0$.  The polynomial is
defined up to multiplication by a non-zero scalar.

It was also shown that an isostatic component  $G^{*}$ of $G$, with at least four vertices, 
creates a polynomial factor of the
pure condition $c_{G}$ which is the pure condition $c_{G^{*}}$ of its subframework.  
So the pure condition of an isostatic
block and hole polyhedron $\P$ would have the form 
\begin{equation}\label{eqn:purecondition}
c_{G(\P)}  =  T(\P) \prod_{B_i \in \B_{\poly}}(c_{B_{i}})  
\end{equation}

for some {\it surface polynomial } $T(\P)$ which does not depend on the specific isostatic framework 
inserted into the
blocks. 

Similarly, for the swapped block and hole polyhedron $\ov{\P}$, which is also isostatic, we have 
\begin{equation}\label{eqn:purecondition2}
c_{G(\ov{\P})}  =  T(\ov{\P}) \prod_{H_i\in \Ho_{\spoly}}(c_{H_{i}}) 
\end{equation}
where $c_{H_{i}}$ represents the pure condition of the isostatic framework inserted into the dual blocks
$\ov{B}_{i} = H_{i}$.   

If we happen to change the isostatic subframework for the  block
$B_{i}$, its factor will change, but $T(\ov{\P}) $ will not, even if we add additional vertices inside the
block.   
\\

%%%%%Example 5.2.1 %%
\begin{example}  Consider the elementary example $\mathcal{C}$ in Figure~\ref{fig:polyhedraA},
in which the block is labelled $v_{1} v_{2}v_{3} v_{4}$ and  the hole is labeled $v_{5} v_{6}v_{7}
v_{8}$.  For simplicity in the following formulas, we re-label the hole vertices as follows: $v_{5}=u_{1}, v_{6}=u_{2}, v_{7}=u_{3},v_{8}=u_{4}$.    
Writing $[\boldsymbol{abcd}]$ for the determinant of the $4\times 4$ matrix of projective
coordinates of the four points $\boldsymbol{a,b,c,d}$, the pure condition has the
form: 
$$
c_{\mathcal{C}} = [\boldsymbol{v_{1}v_{2}v_{3}v_{4}}]
\big(\prod_{i=1 .. 4}[v_{i}v_{i+1}u_{i}u_{i+1}] -\prod_{i=1 .. 4}[v_{i}v_{i-1}u_{i}u_{i+1}]\big)
$$
By convention, we are cycling the indexes so $v_{i+1}= v_{1}$ when $i=4$.  Notice the factor $[v_{1} v_{2}v_{3} v_{4}]$, for that block, which says that the complete graph $K_{4}$ is isostatic unless the four points are coplanar. 
If we consider the swapped block and hole polyhedron, with the block at $v_{5}v_{6}v_{7}v_{8}$, then
(by symmetry)
$$
c_{\mathcal{\ov{C}}} = [\boldsymbol{v_{5} v_{6}a_{7} v_{8}}]
\big(\prod_{i=1 .. 4}[v_{i}v_{i+1}u_{i}u_{i+1}] -\prod_{i=1 .. 4}[v_{i}v_{i-1}u_{i}u_{i+1}]\big)
$$
In particular, we see $T(\mathcal{C}) = T(\ov{\mathcal{C}})$.  
\end{example}

We believe this is typical of
conditions for isostatic block and hole polyhedra and their swapped polyhedra. 
>From our results here, we do know that any configuration $\p$ for which $(\P,\p)$ has a self-stress 
will also provide a self stress on $(\ov{\P}, \p)$ a self-stress.  Equivalently, 
$T(\P,\p) = 0$ if and only if $T(\ov{\P},\p) = 0$,  over all configurations $\p$
 of points in the real projective space.  We
have a stronger conjecture which is compatable with this observation. 

\begin{conjecture}  Given a generically isostatic block and hole polyhedron $\P$,
the surface polynomial of $\P$ is the same as the surface polynomial of the swapped polyhedron
$T(\P) = T(\ov{\P})$
\end{conjecture}\label{conj:purecondition}

Even adding a gusset to  $\P$ to create $\P'$ will show up as factoring in these pure conditions. 
The  $3$-valent insertion at a hole of $\P$, (or onto the modified block in $\P'$) leaves the
residual framework isostatic, so this gusset creates a small factor $[u_{i-1}u_{i}u_{i+1}v_{i}]$ 
where $u_{i-1}u_{i}u_{i+1}v_{i} $ are the four points that make up
the gusset.  When we compare the surface polynomials of $\P$ and $\P'$, we see that
$T(\P') = [\boldsymbol{u_{i-1}u_{i}u_{i+1}v_{u_{i}}}] T(\P)$ if the gusset was added at a hole. For a block gusset, one has this factor as well an additional substitution of $v_{{i}}$ in place of $u_{i}$, in the block polynomial factor.   In particular, these steps make the same change to the surface polynomials $T(\P)$ and $T(\ov{\P})$.  That is, they will be the same after the addition of a gusset if and only if they were the same polynomial before. 

More generally, we anticipate that the pure conditions of block and hole polyhedra will have 
 nice algebraic geometric
forms which would be interesting to explore.   For example, if a block and hole polyhedron has four quadrilateral faces
 $F^{1}, F^{2}, F^{3}, F^{4}$ which are blocks or holes, 
 and if all choices of two of these for blocks and two of these
 for holes generated a generically isostatic block and hole polyhedron, then we can ask whether
 there is a single `surface condition' which is shared by all six possible choices?  This type of question
 is a subject for further research.

%%%%%%%%%%%%
\subsection{More general spheres} \label{sec:moreGeneralSpheres}

So far we have assumed the combinatorial structure was a 3-connected sphere.  
The actual proofs do not require the $3$-connectivity, just the topology of a spherical polyhedron. 
The results would work perfectly well for $2$-connected planar graphs which have identified holes, blocks and
 `surface faces'.  One point requiring some care for 2-connected  polyhedra is that when the surface faces are triangulated, we do not
 accidently insert the same pair of vertices as an `edge' in two different faces, making $G^{M}(\P)$ a multi-graph.   
  In order not to revert to a $3$-connected graph when doing this triangulation, some of the 
  faces at the $2$-disconnection will need to be holes or blocks.   On the other hand,
  by inserting gussets into blocks and holes, one can convert to structure
  into an equivalent (for statics and kinematics) $3$-connected sphere.
  
   A second point is the possibility of two faces sharing two edges.  If these are both blocks, 
 we will have two dual hinges joining the two blocks, so $G^{M}(\P)$ is no longer a graph and the
 two blocks are locked as one rigid object.  
 
    Under these small changes, 
   the results will generalize to  block and hole polyhedra on $2$-connected spheres.  

%%%%%%%%%%%%%%%%%%%%%
\subsection{Surfaces with other toplogies}\label{sec:OtherTopologies}

All of these results were written with the assumption that we started with a spherical polyhedron.
The key to these proofs is actually broader than
just the spherical polyhedral topology.

In Proposition~\ref{prop:motionsToStressesSeparated}, the proof from the motion assignment to the self-stress relied centrally on the fact that at each edge there were two faces of the original structure (including the blocks and holes) and that at each `interior' vertex there was a face-edge cycle, as well as a cycle surrounding each hole-face.  These conditions are 
closely related to the definition of an oriented surface as a complex of faces, edges and vertices.
They are also sufficient to generate a self-stress from the motion assignment.  So  
Proposition~\ref{prop:motionsToStressesSeparated} extends to generalized block and hole 
polyhedra on arbitrary closed oriented polyhedral surfaces.  

If we reread the proof of  Proposition~\ref{prop:motionsToStressesSeparated}, it becomes clear that what was used was that there was an underlying surface (faces, vertices, edges) and that each vertex was surrounded by an face-edge cycle.  So the result transfers to surfaces of any genus.  A motion assignment on a generalized block and hole surface becomes a self-stress of the swapped structure as a bar and joint framework. 

The extensions using gussets (Propositions~\ref{prop:gussetIsomorphismHole},\ref{prop:gussetIsomorphismBlock} and Lemma~\ref{lem:gussetCommutative}) also apply to these generalized block and hole polyhedra, giving isomorphic spaces before and after adding gussets both for stresses and for motions.   This extends the transformation to generalized  block and hole polyhedra on closed oriented polyhedral surfaces.

For other oriented surfaces beyond the sphere, the converse Proposition~\ref{prop:stressesToMotionsSeparated} does not extend.  For example, a triangulated torus has  $|E|=3|V|$, and the are always at least a six dimensional space of self-stresses, including for the space of generic realizations which are first-order rigid. 
For a torus, any vertex-face cycle which is not homologous to zero will not be guaranteed to have
a zero sum in the self-stress, since it is not a cut set.  However, in a motion assignment, this sum on this cycle must be zero. In a general torus, there are two such cycles which are not
homologous, and these generate all the cycles.   If the sum of the self-stress around each of these cycles is also zero, then the self-stress will induce a motion assignment.   This cycle subspace of self-stresses is isomorphic to the space of motion assignments.   Moving to block and hole polyhedra simply extends this problem with some additional complexity.

%%%%%%%%%%%%%%%%%%%%%
\subsection{Polarity and swapping}

There is a theory which explores the rigidity of polar structures for bar and joint structures in $3$-space, also called {\it sheet structures} \cite{sheet}.  In this transformation, plane-isostatic faces (including triangles) go to vertices, edges go to edges and vertices go to plane isostatic frameworks on the plane containing the polar edges (also called {\it sheets}).  The net result is a polar structure which can be realized as a bar and joint framework.  This is a polar transformation which applies to arbitrary bar and joint frameworks in 3-space, preserving static and infinitesimal rigidity (and therefore taking isostatic frameworks  to isostatic sheet structures).   This is a geometric transformation, not just a combinatorial process.  

We will not give the details here, but this polarity has a direct application to block and hole polyhedra and the swapping we have examined here. When this polarity is applied to a polyhedron with isostatic faces, it produces the dual polyhedron with isostatic faces \cite{sheet}.  When the polarity is applied to a rigid block, it produces a dual rigid block.   Overall, when the polarity is applied to a block and hole polyhedron, it will produce a dual block and hole polyhedron.  When it is applied to the swapped polyhedron, it will produce the swapped polyhedron of the dual block and hole polyhedron. 

%%%%%%%%%%%%%


\begin{thebibliography}{10}
%%%%%%%WWAug16%%%%%%%%
\bibitem{alexandrov}  {\sc  A.D. Alexandrov}:
{\bf Convex polyhedra}, GTI, Moscow, 1950. English translation: Springer, Berlin, 2005.

\bibitem{cauchy}  {\sc  A.L. Cauchy}:  A.L. Cauchy,
Recherche sur les polydres - premier m\'emoire,
{\it  Journal de l'Ecole Polytechnique} {\bf 9} (1813), 66Ð86.

\bibitem{crapowhiteley}  {\sc Henry Crapo and Walter Whiteley}:
{Statics of frameworks and motions of panel structures: a projective geometric
introduction}, {\it Structural Topology} {\bf 1} (1982), pp~43-82.

\bibitem{dehn}  {\sc  M. Dehn}:
\"Uber die Starreit konvexer Polyeder, (in German), {\it Math. Ann. }{\bf 77} (1916),
466-473. 

\bibitem{wfsw}  {\sc Wendy Finbow-Singh and Walter Whiteley}:
{Isostatic Almost Spherical Frameworks via Disc Decomposition},
 preprint, York University, 2007

\bibitem{gluck}  {\sc Herman  Gluck}:
{\it  Almost all simply connected closed surfaces are rigid}
 Lecture Notes in Math. {\bf 438}, Geometric Topology, Springer-Verlag (1975), 225-239.

\bibitem{graverCounting}  {\sc Jack Graver}:
{\bf Counting on Frameworks}
Dolciani Mathematical Expositions, MAA, Washington. 2001.

\bibitem{graver}  {\sc Jack Graver, Brigitte Servatius,  and Hermann Servatius}:
{\bf Combinatorial Rigidity}
Graduate Studies in Mathematics, AMS, Providence. 1993.


\bibitem{saliola-whiteley}  {\sc Franco Saliola and Walter Whiteley}:
{Equivalence of First-order Rigidity for Euclidean, Spherical and Hyperbolic Metrics},
Preprint, York University, 2004.


\bibitem{taywhiteley}  {\sc Tiong-Seng Tay and Walter Whiteley}:
{ Generating isostatic frameworks},
 {\it Structural Topology} {\bf 11} (1985), pp~20-69.

\bibitem{whitehandbook}{\sc N. White}: Geometric applications of the Grassmann-Cayley algebra, 
Chapter 59 in  {\bf Handbook of Discrete and Computational Geometry}, J. Goodman and J. O'Rourke (eds),
 CRC Press LLC, Boca Raton, FL; Second Edition,  (2004).

\bibitem{whiteGrassmann} {\sc N. White}: {A tutorial on Grassmann-Cayley algebra}
   In {\bf Invariant Methods
    in Discrete and Computational Geometry}, N. White, Ed. Kluwer Academic Publisher,
    (1995), pp. 93--106.
 
\bibitem{wwI} {\sc Neil L. White and Walter Whiteley}:
{The Algebraic Geometry of Stresses in Frameworks.} {\it SIAM Journal on
Algebraic and Discrete Methods}{ \bf4}  (1983), 481-511.


\bibitem{handbook}{\sc W. Whiteley}:
{Rigidity and Scene Analysis};  in  {\bf Handbook of Discrete and Computational Geometry,} 
J. Goodman and J. O'Rourke (eds.), (second edition) (2004), 1327-1354.

\bibitem{wchapter} {\sc W. Whiteley}: { Matroids from discrete geometry},
in {\bf Matroid Theory},
J. E. Bonin, J. G. Oxley, and B. Servatius, Eds. American Mathematical
Society, Contemporary Mathematics, 1996, vol. 197, pp~171-313.

\bibitem{sheet} {\sc W. Whiteley}: {Rigidity and polarity I: statics of sheetworks};  {\it Geometriae Dedicata} {\bf 22} (1987), 329-362

\bibitem{infp2} {\sc W. Whiteley}: {Infinitesimally rigid polyhedra II: modified spherical frameowrks},
{\it Trans. A.M.S.} {\bf 306} (1988), 115-139.

\bibitem{bipartite} {\sc W. Whiteley}: {Infinitesimal motions of a bipartite framework}  
{\it Pac. J. Math.} 110 (1984), 233-255.

\bibitem{infp1} {\sc W. Whiteley}: { Infinitesimally rigid polyhedra I: statics of frameworks},
{\it Trans. A.M.S.}, {\bf 285} (1984), 431-465.  


\end{thebibliography}
\end{document}